\documentclass[11pt,a4paper,reqno]{amsart}
\usepackage{amsmath,amsthm,amsfonts,amssymb,bm,wasysym}
\usepackage{comment}
\usepackage{caption}
\usepackage{subcaption}
\usepackage{fullpage,bbm}
\usepackage{graphicx}
\usepackage{tikz,ifthen}
\usepackage{hyperref}
\usetikzlibrary{calc}
\usetikzlibrary{patterns}
\usetikzlibrary{arrows}
\usetikzlibrary{decorations.pathreplacing}

\def\d{{\rm d}}
\def\E{\mathbb{E}}
\def\es{\emptyset}

\def\mc{\mathcal}
\def\ms{\mathsf}
\def\one{\mathbbmss{1}}
\def\Q{\mathbb{Q}}
\def\P{\mathbb{P}}
\def\R{\mathbb{R}}
\def\Z{\mathbb{Z}}

\newcommand{\SIR}{\mathsf{SINR}}

\def\eq{\begin{equation}}
\def\en{\end{equation}}

\def\Q{{\Bbb Q}}

\def\e{{\varepsilon}}

\def\D{\Delta}
\def\a{\alpha}

\def\b{\beta}

\def\e{\varepsilon}

\def\phi{\varphi}

\def\r{\rho}
\def\de{\delta}

\def\s{\sigma}
\def\t{t}

\def\D{\Delta}
\def\L{\Lambda}
\def\G{\Gamma}

\def\P{{\Phi}}

\def\T{\T}

\def\V|{{\Vert}}

\def\d{{\rm d}}
\def\E{\mathbb{E}}
\def\I{{I}}
\def\es{\emptyset}
\def\Ln{\Lambda_n}
\def\one{\mathbbmss{1}}
\def\mc{\mathcal}

\def\ms{\mathsf}
\def\N{\mathbb{N}}
\def\one{\mathbbmss{1}}
\def\P{\mathbb{P}}
\def\R{\mathbb{R}}
\def\Z{\mathbb{Z}}

\theoremstyle{plain}
\newtheorem{theorem}{Theorem}
\newtheorem{proposition}[theorem]{Proposition}
\newtheorem{corollary}[theorem]{Corollary}
\newtheorem{lemma}[theorem]{Lemma}

\theoremstyle{definition}

\theoremstyle{remark}

\begin{document}
\title{Large-deviation principles for connectable receivers in wireless networks} % insert title - use \\ if it requires more than one line.
\author{Christian Hirsch}
\author{Benedikt Jahnel}
\author{Paul Keeler}
\author{Robert Patterson}
\thanks{Weierstrass Institute Berlin, Mohrenstr. 39, 10117 Berlin, Germany; E-mail: {\tt christian.hirsch@wias-berlin.de}, {\tt benedikt.jahnel@wias-berlin.de}, {\tt paul.keeler@wias-berlin.de}, {\tt robert.patterson@wias-berlin.de}.}

\begin{abstract}
We study large-deviation principles for a model of wireless networks consisting of Poisson point processes of transmitters and receivers, respectively. To each transmitter we associate a family of connectable receivers whose signal-to-interference-and-noise ratio is larger than a certain connectivity threshold. First, we show a large-deviation principle for the empirical measure of connectable receivers associated with transmitters in large boxes.  Second, making use of the observation that the receivers connectable to the origin form a Cox point process, we derive a large-deviation principle for the rescaled process of these receivers as the connection threshold tends to zero. Finally, we show how these results can be used to develop importance-sampling algorithms that substantially reduce the variance for the estimation of probabilities of certain rare events such as users being unable to connect.
\end{abstract}

 % insert keywords separated by a semicolon


\maketitle

%\allowdisplaybreaks[4]
%\usepackage{amsfonts,amssymb}
%\usepackage[mathscr]{euscript}
%\topmargin -.75in
%\textwidth 6.5in
%\oddsidemargin -.05in
%\textheight 9.3in
%\parindent=8mm
%\frenchspacing
%\parskip8pt

%%%%%%%%%%%%%%%%%%%%%%%%%%
%%%%%%%%%% PK Macros      
%%%%%%%%%%%%%%%%%%%%%%%%%%%%

%% colours %%%

%\begin{document}

%%%% separate workplan  sections  %%%% 

\section{Model description and main results}
\label{intrSec}
We consider a stochastic-geometry model for a wireless network consisting of a family of transmitters and a family of receivers. Transmitters and receivers are modeled by independent homogeneous Poisson point processes $X$ and $Y$ in $\R^d$ whose intensities are assumed to be non-zero and finite and will be denoted by $\lambda_{\ms{T}}$ and $\lambda_{\ms{R}}$, respectively. For instance, we may think of transmitters and receivers as users participating in a device-to-device communication where messages need not be routed via a base station. It is believed that this form of communication will be a central concept in next-generation wireless networks~\cite{5GD2D}. 
The most basic requirement in the design of such networks is to guarantee satisfactory quality of service on average. Additionally, it is desirable to control and quantify the probability of low quality of service to occur. This necessitates a more detailed probabilistic analysis and the theory of large deviations provides the appropriate tools.
%The point process of receivers is assumed to be given by a homogeneous Poisson point process $Y$ that has intensity $\lambda_{\ms{R}}\in(0,\infty)$ and is independent of $X$.
%we consider a homogeneous Poisson point process $Y$ with intensity $\lambda_{\ms{R}}\in(0,\infty)$ acting as possible receivers of messages sent out by a user located at the origin. Additionally, $X$ denotes a Poisson point process of transmitter that is independent of $Y$ and has intensity $\lambda_{\ms{S}}\in(0,\infty)$. The point process $X$ is the cause for the interference experienced at the receiver locations $Y$.

Let us now describe the communication model. 
In order to determine the connection quality of messages sent out from a transmitter located at $x\in\R^d$ to a receiver located at $y\in\R^d$, the \emph{signal-to-interference-and-noise ratio} ($\SIR$) has been identified to be of fundamental importance~\cite{baccelli2009stochastic1}. More precisely, we assume that signals are transmitted with some positive powers $P_x$ and decay according to the path-loss function $\ell(|x-y|)$, where $|\cdot|$ denotes the Euclidean norm in $\R^d$ and $\ell:[0,\infty)\to[0,\infty)$ is a decreasing function satisfying $\ell(r)\in o(r^{-\alpha})$ for some $\alpha>d$. In particular $\ell$ is bounded and $\int \ell(|x|)\d x<\infty$.
%$\int_0^\infty r^{d-1}\ell(r)\d r<\infty$ and .
%$\int_0^\infty r^{d-1}\ell(r)\d r<\infty$ and $\ell(r)\in o(r^{-d})$.
%where $\ell(r)=\min\{1,r^{-\alpha}\}$ for some \emph{path-loss exponent} $\alpha>d$. Our techniques also apply to more general functions with polynomial decay of order larger than $d$, but considering a specific representative simplifies the exposition. In addition to the decay of the signal strength over distance, interferences caused by signals from other nearby users can obstruct successful message transmission.
%The most popular choice is to use the Euclidean norm $|\cdot|=|\cdot|_2$, but the supremum norm $|\cdot|=|\cdot|_\infty$ is often more convenient for mathematical analysis. 
In addition to the deterministic decay over distance, the signal strength is also influenced by random fading effects that are encoded in a positive random variable $F_{x,y}$. Such fading effects can for example come from large obstacles in the environment or multi-path interference due to moving reflectors \cite[Chapter 22]{baccelli2009stochastic2}.

Furthermore, considering a signal sent out from $X_i$, the strength of the \emph{interference} experienced at a location $y\in\R^d$ is assumed to be of the form
$$\I(X_i,y)=\I(X_i,y,X)=w+\sum_{j\ne i}P_{X_j}F_{X_j,y}\ell(|X_j-y|).$$
In words, the interference strength at a given location $y$ consists of a contribution from the thermal noise $w>0$ and the aggregated signal strengths coming from all other transmitters. For notational convenience, we differ from the common convention~\cite{baccelli2009stochastic1} and include the thermal noise $w$ in the interference term.
Hence, the SINR for the transmitter $X_i\in X$ and the possible receiver location $y\in\R^d$ is defined as the ratio of the signal strength by the interference, i.e.,
$$
\SIR(X_i,y)=\SIR(X_i,y,X)=\frac{P_{X_i}F_{X_i,y}\ell(|X_i-y|)}{\I(X_i,y)}.
$$
%where
%$$\I(X_i,y)=\I(X_i,y,X)=\I(y,X)-\ell(|X_i-y|).$$
We assume that a connection can be established between $X_i\in X$ and $Y_j\in Y$ if $\SIR(X_i,Y_j)\ge\t$ for some fixed connectivity threshold $\t$.
The importance of the SINR stems from Shannon's law in information theory, which provides an explicit formula expressing the maximum possible data throughput in terms of $\SIR$, see~\cite[Chapter 16]{baccelli2009stochastic2}.

\input{modPic.tex}

\medskip
In the present paper, we analyze how connectivity properties of the $\SIR$-based network model described above behave in certain asymptotic regimes. First, we associate to each transmitter $X_{i}$ the family of \emph{receivers $Y^{(i)}$ that are connectable to $X_i$}, i.e., 
$$Y^{(i)}=\{Y_j\in Y:\,  \SIR(X_i,Y_j)\ge\t\}.$$ 
An illustration of the transmitters together with their connectable receivers is shown in Figure~\ref{modFig}.
The family $Y^{(i)}$ can be used to express a variety of \emph{frustration events}  for the transmitter $X_i$. For instance $\{Y^{(i)}=\es\}$ describes the frustration event that the transmitter $X_i$ is isolated, in the sense that it fails to communicate with any of the receivers. Similarly, if $B_r(X_i)$ denotes the open Euclidean ball with radius $r$ centered at $X_i$, then $Y^{(i)}\subset B_r(X_i)$ encodes the event that $X_i$ can only communicate with receivers at distance at most $r$. 

%More generally, if $A\in\mc{N}$ is any point-process event ({\bf to be made more precise!}), then we say that $X_i$ is \emph{$A$-frustrated} if $Y^{(i)}-X_i \in A$.
%that it cannot transmit messa
%that are related to the network quality as perceived by the transmitter $X_i$.
%investigate the proportion of transmitters in large boxes that are isolated in the sense that they fail to communicate with any of the receivers. To be more precise, we say that $X_i$ is \emph{isolated} if 
%\begin{align}
%\label{isoDef}
%U_{i,j}>q(\SIR(X_i,Y_j))\quad\text{ for all }j\ge1,
%\end{align}
%Now, let 
%$$N_{n}=\#\{X_i\in \Lambda_n:\, X_i\text{ is $A$-frustrated}\}$$
%$$p_A=\P(Y^{(0)}\in A)$$
%denotes the probability that a typical transmitter (located at the origin $o\in\R^d$) is $A$-frustrated. 

\medskip
Before we state our first main result, let us introduce the precise assumptions on the transmission powers and fading variables.
We assume that the transmission powers $\{P_x\}_{x\in\R^d}$ form an iid random field whose existence is guaranteed by Kolmogorov's extension theorem. 
Note that only the subset of powers $\{P_{X_i}\}_{i\ge1}$ is relevant, but it is notationally convenient to work with the random field indexed by the full space $\R^d$. A similar remark holds for the random fading field $\{F_{x,y}\}_{x,y\in\R^d}$. 
It can reproduce two different kinds of fading effects. First, a contribution stemming from a suitable random environment such as slow fading, which is typically spatially correlated. Second, effects such as fast fading, that are idiosyncratic to the pair $(x,y)$ and therefore do not exhibit spatial correlation. To be more precise for the first contribution, we assume $Z$ to be a homogeneous Poisson point process with intensity $\lambda_{\ms{E}}>0$ modeling the random environment. Moreover, we use an iid random field $\{U_{x,y}\}_{x,y\in\R^d}$ consisting of random variables uniformly distributed on $[0,1]$ for the idiosyncratic effects. Then the random fading field can have the following general form $$F_{x,y}=\Phi(y-x,Z-x,U_{x,y})$$ where $\Phi$ is measurable and positive. In particular, the construction is such that the fading field is spatially translation invariant, i.e., $\{F_{x+z,y+z}\}$ is equal in distribution to $\{F_{x,y}\}$ for any $z\in\R^d$.

The dependence of $\Phi$ on its second component should be local in the sense that there exists an increasing function $s_{\ms{env}}:[0,\infty)\to[0,\infty)$ such that $\Phi(z,\varphi,u)=\Phi(z,\varphi\cap B_{s_{\ms{env}}(|z|)}(o),u)$, where $B_{s_{\ms{env}}(|z|)}(o)$ denotes the Euclidean ball of radius $s_{\ms{env}}(|z|)$ centered at the origin.
Moreover, letting $U$ be a single uniformly distributed random variable on $[0,1]$, we assume that there exist $N>0$, $s_{\ms{max}}>s_{\ms{min}}>0$ such that for any $z\in\R^d$ and any locally finite $\varphi\subset\R^d$ the distribution function $q_{z,\varphi}:t\mapsto \P(\Phi(z,\varphi,U)^{-1}\le t)$
\begin{enumerate}
\item is globally Lipschitz
%, globally Lipschitz in its first derivative, both 
with Lipschitz constant $N$,
\item $q_{x,\varphi}(s)=0$ for $s\le s_{\ms{min}}$ and $q_{x,\varphi}(s)=1$ for $s>s_{\ms{max}}$.
\end{enumerate}
The second condition ensures that the fading variables have support bounded away from zero and infinity. We assume the same for the power variables $P_x$.
Moreover the random objects $X$, $Y$, $Z$, $\{P_x\}$ and $\{U_{x,y}\}$ are independent. 
%The bounds on the transmission powers are identical to the ones of the fading variables for notational convenience only.

We provide an example illustrating possible fading fields within the above framework. For instance a Boolean model $\Xi=\bigcup_{Z_i\in Z}B_1(Z_i)$ can be interpreted as randomly distributed obstacles in a city. If the line of sight between transmitter $x$ and receiver $y$ is blocked by some building 
%$B_1(Z_i)$
the signal propagation is diminished. That is, $F_{x,y}=\exp\big(-\one_{[x,y]\cap\Xi\neq\emptyset} \big)J^{-1}(U_{x,y})$ where $J$ is a globally Lipschitz distribution function of a random variable which is bounded away from zero and infinity. We note that for modeling urban environments it is important to take into account the effects of correlated fading variable due to fixed obstacles. See also~\cite{bacZhang}.

\medskip
Our first main result will provide a \textit{large-deviation principle} (LDP) for the empirical measure of the family of all connectable receivers $Y^{(i)}-X_i$ such that $X_i$ is contained in the box $\Lambda_n=[-n/2,n/2]$ for large $n$. 
To make this precise, we first note that each $Y^{(i)}$ is a random variable in the measurable space $(\mathbf{N}_{\ms{f}},\mc{N}_{\ms{f}})$. 
Here $\mathbf{N}_{\ms{f}}$ is the family of all finite subsets of $\R^d$ that is endowed with the $\sigma$-algebra $\mc{N}_{\ms{f}}$ generated by maps of the form $\ms{ev}_B:\varphi\mapsto\#(\varphi\cap B)$, for any Borel set $B\subset\R^d$. In fact, $\bf{N}_{\ms{f}}$ is also a Polish space, see~\cite[Section A.2.5]{pp1}.
Now, knowing the distribution of the \textit{empirical measure}
$$L_n=\frac{1}{|\L_n|}\sum_{X_i\in\Lambda_n}\delta_{Y^{(i)}-X_i}$$
we can answer questions such as:
\begin{itemize}
\item What is the probability that, when spatially averaged, a certain proportion of transmitters in $\L_n$ are isolated?
\item What is the probability that, when spatially averaged, a certain proportion of transmitters in $\L_n$ have $l$ receiver in an $r$ proximity?
%re isolated?
%to find a certain (spatially averaged) proportion of transmitters in $\Lambda_n$ that are isolated?
%\item What is the probability to find a certain (spatially averaged) proportion of transmitters in $\Lambda_n$ that 
\end{itemize}
 Apart from these examples, $L_n$ can be used to describe more general events like an average number of connectable receivers per transmitter, i.e., $|\L_n|^{-1}\sum_{X_i\in\Lambda_n}\#Y^{(i)}$.

The empirical measure $L_n$ is a random variable with values in the measurable space \linebreak $(\mc{M}_{\ms{f}}(\mathbf{N}_{\ms{f}}),\mc{B}^{\ms{cy}}(\mc{M}_{\ms{f}}))$. Here, $\mc{M}_{\ms{f}}(\mathbf{N}_{\ms{f}})$ denotes the family of all finite measures on $\mathbf{N}_{\ms{f}}$ and $\mc{B}^{\ms{cy}}(\mc{M}_{\ms{f}})$ is the $\sigma$-algebra generated by the evaluation maps $\mu\mapsto \mu(B)$, where $B$ is any bounded Borel set of $\mathbf{N}_{\ms{f}}$. Since our first main result provides a level-2 LDP, the $\tau$-topology on $\mc{M}_{\ms{f}}$ will play an important role. This topology is generated by the maps $\mu\mapsto \mu(B)$ where $B$ is any bounded Borel set of $\mathbf{N}_{\ms{f}}$. We refer the reader to~\cite[Section 6.2]{dz98} for a detailed discussion of this topological space.

%average number of isolated transmitters. 
%denote the number of $A$-frustrated transmitters in the box $\Lambda_n=[-n/2,n/2]^d$. %Observe that by the Slivnyak-Mecke formula, $\E N^A_n=\lambda_{\ms{T}}|\Lambda_n|p_A$, where 

%

\medskip
The LDP allows us to quantify the decay of probability for events away from their ergodic limit on an exponential scale. 
%In a nutshell, LDPs give upper and lower bounds for the probability of unlikely events when the volume grows. 
The exponential rate of decay to zero is proportional to the volume and the proportionality factor is called the rate function. 
%
%is represented by a non-negative rate function $I$ and 
%$$\P_n(A)\approx \exp(-nI(A)).$$
%
In order to identify the LDP rate function, we first recall the notion of \emph{specific entropy of point marked random  fields}. We follow the presentation in~\cite{georgii2} and refer the reader also to~\cite[Chapter 15]{georgiiBook} for further details. 
%First, we write $\mc{P}$ for the family of all \emph{point random fields} in $\R^d$, i.e., the family of all distributions of simple point processes in $\R^d$. 
Let $E$ be a Polish space and write $\mc{E}$ for the corresponding Borel $\sigma$-algebra.
Furthermore, let $\mathbf{N}_E$ denote the family of all configurations $\varphi\subset\R^d\times E$ whose projection to $\R^d$ is injective and with image forming a locally finite set. The space $\mathbf{N}_{E}$ is endowed with the smallest $\sigma$-algebra for which all evaluation maps $\varphi\mapsto\#(\varphi\cap(B\times F))$ are measurable for any Borel sets $B$, $F$ of $\R^d$ and $E$, respectively. Any probability measure on $(\mathbf{N}_E,\mc{N}_E)$ is called $E$-marked point random field.
Let $n\ge1$ and $P$ be an $E$-marked point random field whose realizations are contained in $\Ln$ with probability $1$. Moreover, let $Q$ be another $E$-marked point random field that is absolutely continuous with respect to $P$, where $f$ denotes the respective density. Then, the \emph{specific entropy} $H(Q|P)$ of $Q$ with respect to $P$ is defined as
$$H(Q|P):=P(f\log f),$$
where $P(f\log f)$ denotes the expectation of $f\log f$ with respect to $P$.
This definition is extended to random point fields $Q$ that are not absolutely continuous with respect to $P$ by putting $H(Q|P)=\infty$. Finally, if $P$, $Q$ are any $E$-marked point random fields we introduce the notation
$$h(Q|P):=\sup_{n\ge1}\frac{1}{|\L_n|}H\big(Q_{\Ln}|P_{\Ln}\big),$$
where $P_{\Ln}$, $Q_{\Ln}$ denote the projection of $P$, $Q$ to $\Ln$.
%$P^{(\ms{T})}$ is the law of the Poisson point process of transmitters and 

In the following, we write $\mc{P}_\theta$ for the family of all stationary $E$-marked point random fields of finite intensity. Here, a stationary $E$-marked point random field is a probability measure on $\mathbf{N}_E$ that is invariant with respect to shifts on $\R^d$. The intensity of $Q$ is defined as 
$$\int_{\mathbf{N}_E}\#\{(x_i,e_i)\in\varphi:\,x_i\in[0,1]^d\}Q(\d\varphi).$$
We also need the notion of the {Palm version} of a stationary point random field as defined for example in~\cite{mecPalm}. The (unnormalized) \emph{Palm mark measure} $Q^o$ associated with $Q\in\mc{P}_\theta$ is given by
$$Q^o(F)=\int_{\mathbf{N}_E}\#\{(x_i,e_i)\in\varphi:\,(x_i,e_i)\in[0,1]^d\times F\}Q(\d\varphi),\qquad F\in\mc{E}.$$
%\sum_{(x_i,e_i)\in \varphi}\one\{x_i\in[0,1]^d\}\one_{\{e_i\in F\}}Q(\d \varphi)where $\varphi=\{(x_i,e_i)\}_{i\ge1}$ is a realization of a point process distributed according to $Q$ having intensity $\la$. 
%and where we put $\varphi-x_i=\{(x_j-x_i,e_j)\}_{j\ge1}$.  
In other words, after normalization, $Q^o$ describes the distribution of the marks of $Q$.
%that is obtained by viewing $X$ from the perspective of a randomly selected reference point in $X$.

The concept of random marked point random fields is very flexible so that the probability space associated with $X$, $Y$, $Z$, $\{P_x\}$ and $\{U_{x,y}\}$ can be encoded in this framework, see Section~\ref{indEmpSec} for details.

\medskip
Let us state the first main result of this paper, an LDP for the empirical measure of connectable receivers associated with transmitters in a large box. Starting from a stationary point random field $\Q$ of transmitters, receivers and environment, we define $\Q^*$ as the Palm mark measure of the stationary $\mathbf{N}_{\mathsf{f}}$-marked point random field defined by $\{(X_i,Y^{(i)}-X_i)\}_{i\ge1}$.
%:=\frac{\lambda_{\ms{T}}+\lambda_{\ms{R}}}{\lambda_{\ms{T}}}\Q^o_Z(B\cap\{o\text{ is transmitter}\})$ 
%as the Palm distribution with respect to transmitters where $\Q^o_V$ is the Palm distribution of the marked process as introduced in Section~\ref{intrSec}.
%Note that we are only interested in the Palm measure where the origin is a  transmitter, a more formal exposition in terms of a marked process is given in Section~\ref{indEmpSec}.
\begin{theorem}
\label{ldpThm}
%Let $A\in\mc{N}$ be arbitrary. 
The random measures $\{L_n\}_{n\ge1}$ satisfy an LDP in the $\tau$-topology with rate $|\Lambda_n|$ and good rate function 
$$\mc{I}(Q)=\inf_{\substack{\Q\in\mc{P}_\theta\\  \Q^*=Q}} h(\Q|\P).$$
That is for all $A\in \mc{B}^{\ms{cy}}(\mc{M}_{\ms{f}})$
$$\limsup_{n\to\infty}\frac{1}{|\Lambda_n|}\log\P(L_n\in A)\le-\inf_{Q\in \bar A}\mc{I}(Q)$$
%for all closed $F\subset \mc{M}_{\ms{f}}(\mathbf{N}_{\ms{f}})$ 
and
$$\liminf_{n\to\infty}\frac{1}{|\Lambda_n|}\log\P(L_n\in A)\ge-\inf_{Q\in A^o}\mc{I}(Q)$$
where $\bar A$ denotes the closure and $A^o$ the interior of $A$ respectively. 
%for all open $G\subset \mc{M}_{\ms{f}}(\mathbf{N}_{\ms{f}})$, and t
Moreover, the function $\mc{I}$ is lower semi-continuous and has compact level sets.
%s\log\frac{s}{\theta}+(1-s)\log\frac{1-s}{1-\theta}.$$
%where $Q^o$ denotes the Palm version of $Q$ and the infimum is over all laws $Q$ of stationary point processes satisfying $Q^o(o\text{ isolated})=s$.
\end{theorem}
To prove Theorem~\ref{ldpThm}, we make use of the level-3 LDPs established in~\cite{georgii2} (see also~\cite{georgii1,georgii3} for related results). However, the long-range dependencies induced by the interferences prevent us from applying the contraction principle directly. Similarly to~\cite{adCoKo}, we first have to perform a truncation step and consider an approximate model with finite-range dependencies. In order to deduce Theorem~\ref{ldpThm} from the level-3 LDP in the truncated scenario, we show that by a suitable choice of the truncation range, the truncation error becomes arbitrarily small.

\medskip
In our second main result, we investigate how the connectable receivers associated with a typical transmitter located at the origin behave as the connection threshold $\t$ tends to zero.
%probabilities $q(\cdot)$ 
 Since this scenario turns out to be more complicated than the one considered in Theorem~\ref{ldpThm}, we impose stronger additional assumptions. To be more precise, we assume that $\ell(r)=\min\{1,r^{-\alpha}\}$ for some $\alpha>d$, the transmission power at the origin is fixed (say equal to $1$) and that there is no random environment $Z$. That is, $\{F_{x,y}\}_{x,y\in\R^d}$ are iid and we put $q(a)=\P(F_{x,y}^{-1}\le a)$. 
%make stronger assumptions on the random environment. To be more precise, we additionally assume that there exists $s_{\ms{min}}>0$ such that $q_{x,\varphi}(s)=0$ if $s\le s_{\ms{min}}$ and that $\inf_{x\in\R^d}\P(q_{x,Z}(s))>0$ if $s>s_{\ms{min}}$.
%
%Moreover, we assume that there exist $N>0$, $s_{\ms{max}}>s_{\ms{min}}>0$ such that for any $x,y\in\R^d$ and any locally finite $\varphi\subset\R^d$ the function $q_{x,\varphi}:t\mapsto \P(\Phi(x,\varphi,U_{x,y})^{-1}\le t)$.
%
Moreover, we assume that there exist $N>0$ and $s_{\ms{min}}>0$ such that
\begin{enumerate}
\item $q_{}$ is globally Lipschitz and globally Lipschitz in its first derivative, both with Lipschitz constant $N$,
\item $q_{}(s)=0$ for $s\le s_{\ms{min}}$ and $q_{}(s)>0$ for $s>s_{\ms{min}}$.
\end{enumerate}
%Now, we consider the family of functions $(q_\t)_{\t\in(0,1)}$, where $q_\t(s)=q(\t^{-1} s)$.
% and introduce the associated isolation probabilities 
%$$p_\t=\P(U_{0,j}> q_\t(\SIR(o,Y_j,X\cup\{o\}))\text{ for all }j\ge1).$$
To begin with, we provide some important preliminary observations: First, we note that the receivers connectable to the origin, namely 
$$Y^{\t}=\{Y_j\in Y:\,\SIR(o,Y_j,X\cup\{o\})\ge\t\},$$ 
form a Cox point process with random intensity measure $M_\t$ given by 
$$M_\t(B)=\lambda_{\ms{R}}\int_B\Gamma(\t^{-1}\ell(|y|),y)\d y,$$
where 
\begin{align}\label{Gamma}
\Gamma(a,y)=\E(q(a\I(y)^{-1})|X)\qquad \text{for }a\ge0 \text{ and }y\in\R^d.
\end{align}
In other words, $\G$ is an expectation with respect to the fading field in the interference. More precisely, it is the conditional expectation on the transmitter process  $X=\{(X_i,P_i)\}$ carrying also the transmission powers as marks.
For instance, this observation implies that the probability for the origin to be isolated is given by $p_\t=\E\exp(-M_\t(\R^d))$ and tends to zero as 
 $\t$ tends to zero.
 %isolation becomes less likely, so that $p_\t\to0$. 
The representation of the isolation probability provides a strong hint that the Varadhan-Laplace technique from the theory of large deviations (see e.g.~\cite{georgii2}) could be a useful tool in the analysis of the asymptotic behavior of $p_\t$ as $\t$ tends to zero. In particular, $p_\t$ should decay exponentially as $\t$ tends to zero. The exact form of this decay is presented in Corollary~\ref{lowTauProbCor}. In Theorem~\ref{lowTauProbThm}, we give a more general result describing the exponential decay of unlikely numbers of connectable receivers in space. 
Throughout the entire manuscript, $\beta=1/\alpha$ denotes the inverse of the path-loss exponent. 
Furthermore, we put $\Lambda'_\t=\Lambda_{2(ws_{\ms{min}}\t)^{-\beta}}$, so that $q(t^{-1}\SIR(o,y))=0$ if $y\not\in\Lambda_\t'$.  In the following, we write $\ms{Pois}_s$ for the stationary point random field induced by a homogeneous Poisson point process with intensity $s\ge0$.
%Let us define $$Y^{\t}=\{Y_j\in Y:\, U_{o,j}\le q_\t(\SIR(o,Y_j))\}$$then we have the following LDP.

%\medskip
%Let us also use the short notations $\L(x):=\prod_{i=1}^d[-1/2,x_i]$, $\L_\t(x):=|\Lambda_\t'|^{1/d}\L(x)$ and $\SIR(y):=\SIR(o,y)$. 
%2\t^{-\beta}
\begin{theorem}
\label{lowTauProbThm}
The random measures $\big\{|\Lambda_\t'|^{-1}Y^{\t}(t^{-\beta}\cdot)\}_{\t<1}$ satisfy an LDP in the weak topology with rate $|\L_\t'|$ and good rate function given by 
\begin{align*}
\mc{I}(\phi)=\begin{cases}
  \int_{\L_1'}\mc{I}_y(\dot\phi(y))\d y  & \text{if }\d\phi/\d x=\dot\phi\text{ exists,}\\
  \infty & \text{otherwise,}
\end{cases}
\end{align*}
where
\begin{align}\label{SingleSiteRateFunction}
\mc{I}_y(s)=\inf_{\Q\in\mc{P}_\theta}\big(h(\Q|\P)+h(\ms{Pois}_s|\ms{Pois}_{\hspace{0.05cm}\lambda_{\ms{R}}\Q(\Gamma(|y|^{-\alpha},o))})\big),
\end{align}
and $\Q(\Gamma(|y|^{-\alpha},o))$ denotes the expectation of $\Gamma(|y|^{-\alpha},o)$ for  a stationary marked point process $X=\{(X_i,P_i)\}$ that is distributed according to $\Q$.
\end{theorem}
Note that in contrast to Theorem~\ref{ldpThm} the probability measures $\Q\in\mathcal{P}_\theta$ in \eqref{SingleSiteRateFunction} are distributions only of the transmitters $X$ and their transmission powers $P$. Setting $\phi\equiv 0$ gives the decay of isolation probability, this is the content of the following corollary.

\begin{corollary}\label{lowTauProbCor}
\begin{align*}
\lim_{\t\to0}|\Lambda_\t'|^{-1}\log p_\t&=\lim_{\t\to0}|\Lambda_\t'|^{-1}\log\E\exp\Big(-\lambda_{\ms{R}}\int_{\R^d} \Gamma(t^{-1}\ell(|y|),y)\d y\Big)\\
&=-\int_{\L_1'}\inf_{\Q\in\mc{P}_\theta}\big(h(\Q|\P)+\lambda_{\ms{R}} \Q(\Gamma(|y|^{-\alpha},o))\big)\d y.
\end{align*}
\end{corollary}

Large-deviation principles in SIR-based networks have already been considered in~\cite{ldpInt,ldpGinibre}. However, the question treated in Theorem~\ref{lowTauProbThm} is in a certain sense dual to the ones discussed in~\cite{ldpInt,ldpGinibre}. In those papers a large-deviation principle was derived for the interference at the origin caused by the signals from other users. We investigate a scenario where the origin sends out a signal and we are interested in the interference at the location of the other users. 

The idea for the proof of Theorem~\ref{lowTauProbThm} is to introduce a stationary point process that carries more information than $Y^{\t}$. For this point process, we first establish a level-1 LDP based on the results of~\cite{georgii2}, and then deduce a path-space LDP using the Dawson-G\"artner technique. The proof is concluded by an application of the contraction principle.

Corollary~\ref{lowTauProbCor} shows that $p_\t$ decays exponentially in $\t^{-\frac{d}{\alpha}}$ and provides a variational characterization of the rate function. However, for the purpose of estimating the actual value of $p_\t$, our asymptotic result has two drawbacks. First, in Corollary~\ref{lowTauProbCor}, we do not make any claims as regards to how small $\t$ should be for the asymptotic to be an acceptable approximation. It is not at all clear from the variational formula how to compute (or even approximate) the asymptotic rate function. Nevertheless, when estimating the isolation probability $p_\t$ via Monte Carlo simulations, our large-deviation result can be used to devise an importance-sampling scheme that substantially reduces the estimation variance. In the field of stochastic processes, large-deviation techniques have emerged as a powerful tool to find suitable importance-sampling densities~\cite[Chapter 6.6]{glynn}, but so far have not found widespread use for spatial rare-event problems. 

As a notable exception, we mention~\cite{isSir}, which deals with rare events arising from large values of the interference measured at the origin. In that paper, it is shown that the asymptotically efficient importance-sampling density is given by a certain inhomogeneous Poisson point process. In our setting, the variational characterization in Theorem~\ref{lowTauProbThm} suggests that the asymptotically optimal density is not given by a Poisson point process, but by a collection of location-dependent Gibbs processes. Still, in a first step, we provide simulation results illustrating that using an isotropic Poisson point process already leads to substantial variance reduction. 
Let us also note that importance sampling for Gibbs processes on the lattice has been studied in~\cite{baldiGibbs}. 

The present paper is organized as follows. In Sections~\ref{ldpThmSec} and~\ref{sec3}, we provide the proofs for Theorems~\ref{ldpThm} and~\ref{lowTauProbThm}, respectively. Section~\ref{sec3} also contains the proof of Corollary~\ref{lowTauProbCor}. Finally, in Section~\ref{impSampSec} we describe two importance-sampling schemes and provide some simulation results.
%Finally, we show that these results can be further improved by replacing the Poisson point process by certain repulsive Gibbs point processes.

%\section{Preliminaries on interferences}
%
%In the present section, we prove some basic auxiliary results on interferences that will be used in later parts of the manuscript. First, we show that if interference is small at a certain location $y\in\R^d$, then this remains true for all points in a suitable environment of $y$

\section{Proof of Theorem~\ref{ldpThm}}
\label{ldpThmSec}
As mentioned in Section~\ref{intrSec}, in order to prove Theorem~\ref{ldpThm}, we use the classical level-3 large-deviation result for Poisson point processes~\cite[Theorem 3.1]{georgii2}. However, the interferences induce long-range interactions that are not immediately compatible with the topology $\tau_{\mc{L}}$ of local convergence that is used in~\cite{georgii2}. To resolve this issue, we will proceed similarly to~\cite{adCoKo} and show that a suitable truncation of the path-loss functions appearing in the interference expression induces only a negligible error, see Section~\ref{truncPSec1}. After this truncation, we show in Section~\ref{indEmpSec} how the LDP for the stationary empirical field~\cite[Theorem  3.1]{georgii2} can be used to prove Theorem~\ref{ldpThm}.

%the problem of computing the small-$\t$ limit of the isolation probability $p_\t$ has two distinguished features that do not fit immediately into the framework outlined in~\cite{georgii2}. 

%On the one hand, the interferences induce long-range interactions that are not immediately compatible with the topology of local convergence that is used in the framework of~\cite{georgii2}. To resolve this issue, we will proceed similarly to~\cite{adCoKo} and show that a suitable truncation of the path-loss functions appearing in the interference expression induces only a negligible error, see Section~\ref{truncPSec}. 

%We plan to make use of the level-3 LDPs considered in~\cite{georgii1,georgii2,georgii3}. However, since $N_n$ is not a continuous functional, the proof of Theorem~\ref{ldpThm} will involve substantial additional work. 

\subsection{Truncation of the path-loss function}
\label{truncPSec1}
 First, we show that only an asymptotically negligible error occurs when disregarding transmitters close to the boundary of $\Lambda_n$. This is a well-known consequence of the Poisson concentration property~\cite[Chapter 2.2]{lugosi}, but for the convenience of the reader, we provide a detailed proof. %First, we note that it suffices to consider contributions coming from the inner box $\Lambda_{n-b}$. 

\begin{lemma}
\label{erosionLem}
Let $b,\varepsilon>0$ be arbitrary. Then,% there exists $\delta>0$ such that 
$$\lim_{n\to\infty}\frac{1}{|\Lambda_n|}\log\P(X(\Lambda_n\setminus \Lambda_{n-b})\ge\varepsilon|\Lambda_n|)=-\infty.$$
\end{lemma}
\begin{proof}
Let $\delta=\lambda_{\ms{T}}(1-(1-b/n)^d)$, $m=\delta n^d$ and $\tau=\varepsilon n^d$,
%Choose $\delta>0$ such that $\varepsilon \log (\rho e^2)\le -2K$, where $\rho=\lambda_{\ms{T}}(1-(1-\delta)^n)/\varepsilon$. 
then the Poisson concentration inequality~\cite[Chapter 2.2]{lugosi} implies that
\begin{align*}
\P(X(\Lambda_n\setminus \Lambda_{n-b}))\ge\tau)\le(m/\tau)^{\tau}e^{\tau-n}=(\delta\varepsilon^{-1})^{\varepsilon n^d} e^{(\varepsilon-\delta)n^d}\le\exp(\varepsilon n^d\log(e\delta\varepsilon^{-1})).
\end{align*}
Since $\log(e\delta\varepsilon^{-1})$ tends to $-\infty$ as $n\to\infty$, this proves the claim.
%provided that $\varepsilon\ge2\delta$. In particular, %choosing $\delta>0$ such that 
%%$$n^{-d}\log\P(\#(X\cap(\Lambda_n\setminus \Lambda_{n-2b}))\ge\varepsilon n^d)\le\frac{\varepsilon}{2}\log(e\delta\varepsilon^{-1}),$$
%as required.
%$$2^{-1}\varepsilon\log(e^2\delta'\varepsilon^{-1})\le-K$$
%completes the proof.
\end{proof}

Next, we show that truncating the path-loss function in the interference at a finite threshold only leads to a small error provided that the threshold is chosen sufficiently large. To be more precise, for $b\ge1$ we put $\ell_b(r)=\ell(r)$ if $r<b$ and $\ell_b(r)=0$ if $r\ge b$. Furthermore, we define 
$$\I_b(X_i,y)=w+\sum_{j\ne i}P_{X_j}F_{X_j,y}\ell_b(|X_j-y|),\quad\text{and}\quad\SIR_b(X_i,y)=\frac{P_{X_i}F_{X_i,y}\ell(|X_i-y|)}{\I_b(X_i,y)},$$
and 
$$L^b_n=\frac{1}{|\Lambda_n|}\sum_{X_i\in\Lambda_n}\delta_{Y^{(i),b}-X_i}.$$
 where $Y^{(i),b}=\{Y_j\in Y:\, \SIR_b(X_i,Y_j)\ge\t\}$ denotes the point process of \emph{$b$-connectable receivers} for the transmitter $X_i$.
%$U_{i,j}>q(\SIR_b(X_i,Y_j))$ for all $j\ge1$. Note that we do not truncate the path-loss function in the numerator. In particular, letting 
%$$L^b_n=\#\{X_i\in\Lambda_n:\, X_i\text{ is $(b,A)$-frustrated}\}$$
%denote the number of $(b,A)$-frustrated transmitters in $\Lambda_n$, 
We show that when using the total variation distance
$$d_{\ms{TV}}(L_n,L_n^b)=\sup_{B\in\mc{N}_f}|L_n(B)-L_n^b(B)|,$$
the random measures $\{L^b_n\}_{n\ge1}$ are exponentially good approximations of the random measures $\{L_n\}_{n\ge1}$ in the sense of~\cite[Definition 4.2.14]{dz98}.

\begin{lemma}
\label{bTruncLem}
Let $\varepsilon>0$ be arbitrary. Then,
$$\lim_{b\to\infty}\limsup_{n\to\infty}\frac{1}{|\Lambda_n|}\log\P(\d_{\ms{TV}}(L_n,L_n^b)\ge\varepsilon)=-\infty.$$
\end{lemma}
\begin{proof}
To be specific and for notational convenience let us assume that the support of the power variables is contained in $[s_{\ms{max}}^{-1},s_{\ms{min}}^{-1}]$.
% $s_{\ms{max}}^{-1}\le P_x\le s_{\ms{min}}^{-1}$. 
The definition of the total variation distance implies that 
$$\d_{\ms{TV}}(L_n,L_n^b)\le \frac{1}{|\Lambda_n|}\#\{X_i\in \Lambda_n: Y^{(i)}\ne Y^{(i),b}\}.$$
Next, by Lemma~\ref{erosionLem}, we only need to consider those $X_i$ that are contained in $\Lambda_{n-2r_0}$, where $r_0>0$ is chosen such that $\ell(r_0)\leq ws^2_{\ms{min}}\t$. Then, almost surely, for $X_i\in\Lambda_{n-2r_0}$ and $Y_i\in\Lambda_{n}^c$, $\SIR_b(X_i,Y_j)<\t$ for all $b\ge 1$. Consequently, it suffices to bound the number of transmitter-receiver pairs $(X_i,Y_j)\in X\times Y$ such that  $X_i\in\Lambda_{n-2r_0}$, $Y_i\in\Lambda_{n}$ and  $\SIR(X_i,Y_j)<\t\le \SIR_b(X_i,Y_j)$.
% $X_i$ can $b$-send to $Y_j$, but it cannot $\infty$-send to $Y_j$. 
%Each $Y_j$ is connects to at most $9$ transmitters $X_i$. 
In fact, it suffices to focus on the receivers in these pairs. Indeed, let us call $Y_j$ \emph{$b$-pivotal} if there exists some transmitter $X_i$ such that the pair $(X_i,Y_j)$ has these properties.
Then, since we assumed that $q_{x,\varphi}(r)=0$ for $r\le s_{\ms{min}}$, for each receiver $Y_j$ there exist $K=\lceil \t^{-1}s^2_{\ms{max}}/s^2_{\ms{min}}\rceil$ transmitters $A(Y_j,X)=\{X_{i_1},\ldots,X_{i_K}\}$ such that $\SIR_b(X_{i},Y_j)<\t$ if $X_i\not\in A(Y_j,X)$.
Hence, it suffices to show that for every $\varepsilon>0$,
\begin{align*}
%\label{bTruncEq}
\lim_{b\to\infty}\limsup_{n\to\infty}\frac{1}{|\Lambda_n|}\log\P(\#\{Y_j\in\Lambda_n:\,Y_j\text{ is $b$-pivotal}\}\ge\varepsilon|\Lambda_n|)=-\infty.
%Y_j\text{ is $A$-frustrated but not $(b,A)$-frustrated}\}
\end{align*}
In order to do so, we use the exponential Markov inequality with $s\ge1$ and estimate
\begin{align*}
&\P(\#\{Y_j\in\Lambda_n:\,Y_j\text{ is $b$-pivotal}\}\ge\varepsilon|\Lambda_n|)\\
&\le \exp (-s\varepsilon|\Lambda_n|)\E\exp(s\#\{Y_j\in\Lambda_n:\,Y_j\text{ is $b$-pivotal}\}).
%Y_j\text{ is $A$-frustrated but not $(b,A)$-frustrated}\}
\end{align*}
Hence, it suffices to show that for every $s\ge1$,
%In order to be in a position to apply the Markov inequality, let us compute the exponential moments 
%By the Markov inequality, we should compute exponential moments, that is,
\begin{align*}
\lim_{b\to\infty}\limsup_{n\to\infty}\frac{1}{|\Lambda_n|}\log\E\exp(s\#\{Y_j\in\Lambda_n:\,Y_j\text{ is $b$-pivotal}\})=0.
\end{align*}
The point process of receivers that are $b$-pivotal form a stationary Cox point process with random intensity measure %of this Cox process is given by
%$$M'(B)=\lambda_{\ms{R}}\int_B\Big(1-\prod_{X_i\in\Lambda_{2n}}\big(1-(q(\SIR_b(X_i,y))-q(\SIR(X_i,y)))\big)\Big)\d y,\qquad B\in\mc{B}(\R^d),$$
$$M'(B)=\lambda_{\ms{R}}\int_B\P(\text{$y$ is $b$-pivotal}|X,Z)\d y,\qquad B\in\mc{B}(\R^d),$$
where we think of $X=\{(X_i,P_i)\}_{i\ge1}$ as a marked point process
%\Big(1-\prod_{X_i\in\Lambda_{2n}}\big(1-(q(\SIR_b(X_i,y))-q(\SIR(X_i,y)))\big)\Big)
%since $Y_j$ is $b$-pivotal if and only if $q(\SIR(X_i,Y_j))<U_{i,j}\le q(\SIR_b(X_i,Y_j)))$ for some $X_i\in\Lambda_{n}$. 
%$U_{i,j}\in $
%but not 
%.
%\sum_{X_i\in\Lambda_{2n}}q(\SIR_b(X_i,y))-q(\SIR(X_i,y))\d y,\qquad B\in\mc{B}(\R^d),$$
%as conditioned on $X$, for each $j\ge1$ the expected number of $i\ge1$ such that $q(\SIR_b(X_i,y))<U_{i,j}\le q(\SIR(X_i,y))$ is precisely $\sum_{X_i\in\Lambda_n}q(\SIR_b(X_i,y))-q(\SIR(X_i,y))$. 
and we evaluate the probability with respect to the fading variables associated with the pairs $(y,X_i)_{i\ge1}$.
Since $q$ is assumed to be globally Lipschitz with constant $N$, we arrive at 
\begin{align*}
&\P(\text{$y$ is $b$-pivotal}|X,Z)\\
&\quad\le \sum_{X_i\in A(y,X)}\P(\SIR(X_i,y)<\t\le \SIR_b(X_i,y)|X,Z)\\
&\quad\le \sum_{X_i\in A(y,X)}\P(F_{X_i,y}^{-1}\t^{}\in[P_{X_i}\ell(|X_i-y|)\I(X_i,y)^{-1},\,P_{X_i}\ell(|X_i-y|)\I_b(X_i,y)^{-1}]|X,Z)\\
&\quad\le \sum_{X_i\in A(y,X)}\ell(|X_i-y|)^{}Ns_{\ms{min}}^{-1}\t^{-1}w^{-2}\E(\I(X_i,y)-\I_b(X_i,y)|X,Z),
\end{align*}
%\begin{align*}
%q(\SIR_b(X_i,y))-q(\SIR(X_i,y))&\le L(\I_b(X_i,y)^{-1}-\I(X_i,y)^{-1})\\&\le L(\I(X_i,y)-\I_b(X_i,y))\le L( \I(y)-\I_b(y))
%\end{align*}
which is at most 
$$S\sum_{i\ge1}\ell(|X_i-y|)-\ell_b(|X_i-y|)$$
where $S=K\ell(0)Ns_{\ms{min}}^{-3}\t^{-1} w^{-2}$.
%where $\I(y):=\sum_{X_i\in X}\ell(|X_i-y|)$ and similar for $\I_b(y)$. 
In particular,
%Again, since for each $y\in\R^d$ there can exist at most one transmitter $X_i$ such that $q(\SIR_b(X_i,y))>0$, 
we obtain that 
$$M'(B)\le S'\int_B\sum_{i\ge1}\ell(|X_i-y|)-\ell_b(|X_i-y|)\d y,$$
where $S'=\lambda_{\ms{R}}S$.
Hence, using the formula for the Laplace functional of a Cox point process, we get that 
\begin{align*}
&\E\exp\big[s\#\{Y_j\in\Lambda_{2n}:\,Y_j\text{ is $b$-pivotal}\}\big]\cr
&\quad\le\E\exp\big[(e^s-1)S'\int_{\Lambda_{2n}}\sum_{i\ge1}\ell(|X_i-y|)-\ell_b(|X_i-y|)d y\big]\cr
&\quad=\exp\big[\lambda_{\ms{T}}\int_{\R^d}\exp\big((e^s-1)S'\int_{\Lambda_{2n}}\ell(|x-y|)-\ell_b(|x-y|)\d y\big)-1\d x\big].
%&\quad\le\exp\Big(\int_{\Lambda_{2n}}\exp\Big(s\lambda_{\ms{R}}L\int_{\R^d\setminus B_b(x)}\ell(|x-y|)\d y\Big)-1\d x\Big)\\
%%&\quad\le\exp\Big(\int_{\Lambda_n}\exp\big(sd\kappa_db^{d-\alpha}\big)-1\d x\Big)\\
%&\quad\le\exp\Big(n^d\Big(\exp\Big(s\lambda_{\ms{R}}Ld\kappa_d\int_{\R^d\setminus B_b(o)}\ell(|y|) \d y\Big)-1\Big)\Big),
\end{align*}
Notice, that we can bound the integral 
\begin{align*}
&\int_{\R^d}\exp\big[\tau\int_{\Lambda_{2n}}\ell(|x-y|)-\ell_b(|x-y|) \d y\big]-1\d x\\
&\le\int_{\Lambda_{4n}}\hspace{-0.1cm}\exp\big[\tau\int_{\R^d\setminus B_b(x)}\ell(|x-y|) \d y\big]-1\d x+\int_{\R^d\setminus \Lambda_{4n}}\hspace{-0.1cm}\exp\big[\tau\int_{\Lambda_{2n}}\ell(|x-y|) \d y\big]-1\d x.
\end{align*}
where $\tau:=(e^s-1)S'$.
In the next step, we derive bounds for these expressions separately. For the first, we get that 
$$\frac{1}{n^d}\int_{\Lambda_{4n}}\exp\big[\tau\int_{\R^d\setminus B_b(x)}\ell(|x-y|) \d y\big]-1\d x=4^d\big[\exp\big(\tau\int_{\R^d\setminus B_b(o)}\ell(|y|)\d y\big)-1\big]$$
which tends to zero as $b$ tends to infinity. For the second expression,  we note that for $x\in\R^d\setminus\Lambda_{4n}$ and $y\in\Lambda_{2n}$,
$$|x-y|=\frac{|x-y|+\sqrt{d}|x-y|}{1+\sqrt{d}}\ge\frac{|x-y|+|y|}{1+\sqrt{d}}\ge(1+\sqrt{d})^{-1}|x|.$$ 
Consequently, using that $\ell(r)\in o(r^{-d})$, we have for large $n$
\begin{align*}
&\frac{1}{n^d}\int_{\R^d\setminus \Lambda_{4n}}\exp\big[\tau\int_{\Lambda_{2n}}\ell(|x-y|) \d y\big]-1\d x\\
&\quad\le\frac{1}{n^d}\int_{\R^d\setminus \Lambda_{4n}}\exp\big[\tau(2n)^d\ell((1+\sqrt{d})^{-1}|x|)\big]-1\d x\\
&\quad\le\int_{\R^d\setminus \Lambda_{4n}}\tau2^{d+1}\ell((1+\sqrt{d})^{-1}|x|)\d x=\int_{\R^d\setminus \Lambda_{4n(1+\sqrt{d})^{-1}}}\tau2^{d+1}(1+\sqrt{d})^{d}\ell(|x|)\d x,
\end{align*}
which tends to zero as $n$ tends to infinity.
\end{proof}

%\begin{corollary}
%\label{bTruncCor}
%Let $\varepsilon>0$ be arbitrary. Then,
%$$\lim_{b\to\infty}\lim_{n\to\infty}\frac{1}{|\Lambda_n|}\log\P(N_n-N^b_n\ge\varepsilon|\Lambda_n|)=-\infty.$$
%\end{corollary}
%\begin{proof}
%\end{proof}
\subsection{Application of LDP for the stationary empirical field}
\label{indEmpSec}
In order to apply~\cite[Theorem 3.1]{georgii2}, we need to relate the empirical measure of connectable receivers to the stationary empirical field considered in~\cite{georgii2}. Here the first task consists in encoding the probability space carrying the point processes of transmitters $X$, the point process of receivers $Y$, the random environment $Z$, the transmission powers $\{P_x\}$ and the iid family $\{U_{x,y}\}$ in the framework of stationary marked point processes. To be more precise, we put $\Sigma=\{\ms{E},\ms{R},\ms{T}\}$ and consider the mark space $E=\Sigma\times(0,\infty)\times[0,1]^{\N}$ equipped with some complete and separable metric. Furthermore, we let $V$ denote an independently $E$-marked homogeneous Poisson point process with intensity $\sum_{\sigma\in \Sigma}\lambda_{\sigma}$. The mark distribution on $E$ is a product of three distributions defined on the spaces $\Sigma$, $(0,\infty)$ and $[0,1]^{\N}$, respectively. First, on $\Sigma$, we choose the distribution which assigns $\sigma\in\Sigma$ the probability $\lambda_{\sigma}/(\sum_{\sigma'\in \Sigma}\lambda_{\sigma'})$. Second, on $(0,\infty)$ we choose the distribution of the transmission power $P_x$ considered in Section~\ref{intrSec}.
Third, the distribution on $[0,1]^{\N}$ describes a family of iid random variables that are uniformly distributed $[0,1]$. The Poisson point process $Z$ that generates the random environment is represented by elements of $V=(v_i,\sigma_i,P_i,(U_{i,j})_{j\ge1})_{i\ge1}$ with $\sigma_i=\ms{E}$.
Elements of $V=(v_i,\sigma_i,P_i,(U_{i,j})_{j\ge1})_{i\ge1}$ with $\sigma_i={\ms{T}}$ are thought of as transmitters and are denoted by $X$. Elements of $V=(v_i,\sigma_i,P_i,(U_{i,j})_{j\ge1})_{i\ge1}$ with $\sigma_i={\ms{R}}$ are thought of as receivers and are denoted by $Y$. 
We note that the power variables $P_i$ have not meaning if $\s_i\neq \ms{T}$. The random variables $U_{x,y}$ should be thought of as being attached to the the transmitters.
Moreover, proceeding as in~\cite[Section 1]{georgii2}, let
%$X\cap\Lambda_n$, i.e., the family of transmitters in $\Lambda_n$ and let 
$$V^{\ms{per},n}=\bigcup_{s\in\Z^d}((V\cap\Lambda_n)+ns)$$
denote the periodic spatial continuation of $V\cap\Lambda_n$.
%The periodized process of receivers  $Y^{\ms{per},n}$ is defined similarly. 
The \textit{stationary empirical field} is defined as $$R_{n,V}:=\frac{1}{|\Ln|}\int_{\L_n}\one_{V^{\ms{per},n}-v}\d v$$
where $V^{\ms{per},n}-v=\{(v_j-v,e_j)\}_{j\ge 1}$ is the spatial translation of $V^{\ms{per},n}$ by $v$.
%The \textit{individual empirical field} is defined as $$R^o_{n,V}:=\frac{1}{|\Ln|}\sum_{(v_i,e_i)\in V\cap(\Ln\times E)}\one_{V^{\ms{per},n}-v_i}.$$
%where $V-v_i=\{(v_j-v_i,e_j)\}_{j\ge 1}$ the spatial translation of $V$ by $v_i$. 
Now, we let $Y^{\ms{per},n,b,(i)}$ denote the family of periodized receivers that have a $b$-connection to the transmitter $X_i^{\ms{per},n}=(x_i,\ms{T},P_i,(U_{i,l})_{l\ge1})$. More precisely, 
$$Y^{\ms{per},n,b,(i)}=\Big\{Y_j=(y_j,\ms{R},P_j,(U_{j,l})_{l\ge1})\in Y^{\ms{per},n}:\, \t\le \frac{P_{i}F_{x_i,y_j}\ell(|x_i-y_j|)}{w+\sum_{k\ne i}P_kF_{x_k,y_j}\ell(|x_k-y_j|)}\Big\},$$ 
where 
$$F_{x_i,y_j}=\Phi(y_j-x_i,Z^{\ms{per},n}-x_i,U_{i,\Psi(y_j-x_i,Y^{\ms{per},n}-x_i)}),$$
and where the integer $\Psi(y_j-x_i,Y^{\ms{per},n}-x_i)\ge1$ is defined as follows. If $k\ge 1$ is such that $y_j-x_i$ is the $k$-th closest element in $Y^{\ms{per},n}-x_i$ to the origin, then we put $\Psi(y_j-x_i,Y^{\ms{per},n}-x_i)=k$.
This construction will ensure translation invariance for the periodized version.
%
%
% be the 
%
%such that $y_j-x_i$ is the $\Psi(y_j-x_i,Y^{\ms{per},n}-x_i)$ closest element in $Y^{\ms{per},n}-x_i$ seen from the origin.
The empirical measure $L_n^{\ms{per},b}$ of $b$-connectable receivers associated with transmitters in $\Lambda_n$ when the network is based on periodized configurations 
can also be expressed as a function of $R_{n,V}$. Indeed, by the same technique that was used to define the individual empirical field in~\cite{georgii2}, we arrive at
\begin{align*}
L^{\ms{per},b}_n&=\frac{1}{|\L_n|}\sum_{x_i\in X^{\ms{per},n}\cap\Ln}\delta_{Y^{\ms{per},n,b,(i)}-x_i}=\frac{1}{|\L_n|}\int_{\L_n}g'(V^{\ms{per},n}-v)\d v,
\end{align*}
where 
$$g'(V^{\ms{per},n}-v)=\sum_{x_i-v\in(X^{\ms{per},n}-v)\cap\L_1}g(V^{\ms{per},n}-x_i),$$
and where $g$ is the Dirac measure concentrated on the family of $b$-connectable receivers from the origin  multiplied with the indicator function that the origin is a transmitter.
%=
%\sum_{x_i\in\Ln}\delta_{Y^{\ms{per},n,b,(i)}-x_i
%&=\frac{1}{\lambda_{\ms{T}}|\Lambda_n|}\sum_{V_i=(v_i,\sigma_i,(U_{i,j})_{j\ge1})\in\Ln}\delta_{Y^{\ms{per},n,b,(i)}-v_i}\one\{\s_i=\ms{T}\}.
%\end{align*}
%By the definition of $Y^{\ms{per},n,b,(i)}$ from above it is clear, that each of the summands is of the form $g(V-v_i)$ 
% Note that $g$ is a local function.
%In particular $L^{\ms{per},b}_n=\frac{|\Ln|}{X(\L_n)}R^o_{n,V}(g)$.
%$$N^{\ms{per},b}_n=\{Z_i=(\xi_i,\s_i,(U_{i,j})_{j\ge1}: \xi_i\in\L_n, \s_i=T, \{Z_j\in Z^{\ms{per},n}: \}\}$$
%and denotes the number of $(b,A,\ms{per})$-frustrated transmitters in $\Lambda_n$, i.e., the number of $(b,A)$-frustrated transmitters in $\Lambda_n$ when the network is based on the periodized configuration $X^{\ms{per},n}$ instead of $X$. Note that the periodic extension is also applied to the marks $\{U_{i,j}\}_{i,j\ge1}$. 
Next, we prove that the random measures $\{L^{\ms{per},b}_n\}_{n\ge1}$ and $\{L^b_n\}_{n\ge1}$ are exponentially equivalent (in the sense of~\cite[Definition 4.2.10]{dz98}), when using the total variation metric.
\begin{lemma}\label{PeriodError}
Let $\varepsilon>0$ be arbitrary. Then,
$$\lim_{n\to\infty}\frac{1}{|\Lambda_n|}\log\P(\d_{\ms{TV}}(L_n^{\ms{per},b},L_n^b)\ge\varepsilon)=-\infty.$$
\end{lemma}
\begin{proof}
As in Lemma~\ref{bTruncLem}, choose $r_0\ge1$  such that $\ell(|x-y|)\le ws^2_{\ms{min}}\t$ if $|x-y|\ge r_0$. In particular $Y^{(i),b}\subset B_{r_0}(X_i)$. 
Further by the truncation of the interference, to decide whether $Y_j\in Y^{(i),b}$ it suffices to look at transmitters in $B_b(Y_j)$. As a consequence, the family $Y^{(i),b}$  depends only on the network configuration in $B_{r_0+b+s_{\ms{env}}(b)}(X_i)$.
Hence, $\d_{\ms{TV}}(L_n^{\ms{per},b},L_n^b)\le \#(X\cap \Lambda_n\setminus \Lambda_{n-2(r_0+b+s_{\ms{env}}(b))})$, and the claim follows from Lemma~\ref{erosionLem}.
\end{proof}
%Notices that events of the form $\{N^{\ms{per},b}_n\in G\}$ can be expressed as 

Now, we are in a position to provide a proof for the LDP asserted in Theorem~\ref{ldpThm} when $L_n$ is replaced by $L^b_n$. Let $\Q$ be the distribution of some stationary $E$-marked point process $V=(v_i,\sigma_i,P_i,(U_{i,j})_{j\ge1})_{i\ge1}$. Then, we define $\Q^{*,b}$ as the Palm mark measure of the marked point process $(X_i,Y^{(i),b}-X_i)$. Here, as above, $X_i\in V$ are interpreted as transmitters and $Y^{(i),b}\subset V$ as the $b$-connectable receivers.
%$\Q^o(B):=\frac{\lambda_{\ms{T}}+\lambda_{\ms{R}}}{\lambda_{\ms{T}}}\Q^o_Z(B\cap\{o\text{ is transmitter}\})$ as the Palm distribution with respect to transmitters where $\Q^o_V$ is the Palm distribution of the marked process as introduced in Section~\ref{intrSec}.
\begin{proposition}
\label{ldpBProp}
The random measures $\big\{L^b_n\}_{n\ge1}$ satisfy an LDP in the $\tau$-topology with rate $|\Lambda_n|$ and good rate function 
$$\mc{I}^b:Q\mapsto\inf_{\substack{\Q\in\mc{P}_\theta\\ \Q^{*,b}=Q}}h(\Q|\P).$$
\end{proposition}
\begin{proof}

First, we note that the map $\Phi_b:\Q\mapsto \Q^{*,b}$ is continuous with respect to the $\tau$-topology. Indeed, for $\Q_n\to\Q$ and $B\in\mc{N}_f$, the locality that is established after truncating the interferences gives that $|\Q^{*,b}(B)-\Q^{*,b}_n(B)|\to0$ as $n\to\infty$.
As $\Phi_b(R_{n,V})=L^{\ms{per},b}_n$, we can apply~\cite[Corollary 3.2]{georgii2} and the contraction principle. Thus the random measures $\big\{L_n^{\ms{per},b}\}_{n\ge1}$ satisfy an LDP with good rate function $\mc{I}_b$. 
%Since
%Indeed, in order to apply~\cite[Corollary 3.2]{georgii2} we use the function $g$ from above and define
%\begin{align*}
%G(R^o_{n,Z}):=\begin{cases}
%  \infty  & \text{if }R^o_{n,Z}(g)/\lambda_{\ms{T}}\notin F_m\\
%  0 & \text{otherwise.}
%\end{cases}
%\end{align*}
%Let $F$ be a compact subset of $[0,1]$ and let $F_m$, $m\ge1$ denote the set of all $x\in[0,\infty)$ such that $|x-y|\le1/m$ for some $y\in F$. We can assume $$\limsup_{n\to\infty}\frac{1}{|\Lambda_n|}\log\P\Big(\frac{1}{\lambda_{\ms{T}}|\Lambda_n|}N^b_n\in F\Big)>-\infty$$ 
%since otherwise there is nothing to show. For any $m\ge 1$ we have 
%\begin{align*}
%\P\Big(\frac{1}{\lambda_{\ms{T}}|\Lambda_n|}N^b_n\in F\Big)\le\P\Big(\frac{1}{\lambda_{\ms{T}}|\Lambda_n|}N^{\ms{per},b}_n\in F_m\Big)+\P\Big(|N^{\ms{per},b}_n-N^b_n|\ge\frac{1}{m}\lambda_{\ms{T}}|\Lambda_n|\Big)
%\end{align*}
%and thus by Lemma~\ref{PeriodError} 
%\begin{align*}
%\limsup_{n\to\infty}\frac{1}{|\Lambda_n|}\log\P\Big(\frac{1}{\lambda_{\ms{T}}|\Lambda_n|}N^b_n\in F\Big)\le\limsup_{n\to\infty}\frac{1}{|\Lambda_n|}\log\P\Big(\frac{1}{\lambda_{\ms{T}}|\Lambda_n|}N^{\ms{per},b}_n\in F_m\Big).
%\end{align*}
Finally, Lemma~\ref{PeriodError} shows that $\{L^{\ms{per},b}_n\}_{n\ge1}$ and $\{L^b_n\}_{n\ge1}$ are exponentially equivalent with respect to the total variant distance. This implies exponential equivalence of $\{L^{\ms{per},b}_n\}_{n\ge1}$ and $\{L^b_n\}_{n\ge1}$ when evaluated on an arbitrary Borel subset of $\mathbf{N}_{\ms{f}}$. So the claim follows from~\cite[Corollary 1.10, Remark 1.4]{eiSchm}.
\end{proof}
The same arguments also prove the following result, where we consider the marked point process $(X_i,Y^{(i),b}-X_i,Y^{(i),b'}-X_i)$ at different truncation thresholds $b'>b\ge1$. Starting from $\Q\in\mc{P}_\theta$, the associated Palm mark distribution is denoted by $\Q^{*,b,b'}$.
\begin{lemma}
\label{ldpBProp2}
Let $b'>b\ge1$. Then, the random variables $\big\{\tfrac{1}{|\L_n|}\#\{{X_i\in\L_n}:\, Y^{(i),b'}\ne Y^{(i),b}\}\}_{n\ge1}$ satisfy an LDP with rate $|\Lambda_n|$ and good rate function 
$$s\mapsto \inf_{\substack{\Q\in\mc{P}_\theta\\ \Q^{*,b,b'}(Y^{(o),b'}\ne Y^{(o),b})=s}}h(\Q|\P).$$
\end{lemma}

Finally, we complete the proof of Theorem~\ref{ldpThm}. In Lemma~\ref{bTruncLem}, we showed that $\{L^b_n\}_{n\ge1}$ are exponentially good approximations of 
$\{L_n\}_{n\ge1}$ and hence an application of~\cite[Theorem 1.13]{eiSchm} is natural. 
%However, this result requires that the frustration probability with truncated interferences converges to the non-truncated case uniformly in some set of stationary point processes. This is difficult to implement in our situation. Nevertheless, we are able to overcome this difficulty using uniformness in the truncation parameter $b$.
%
%some uniformness in the sequence of minimizing measures in the rate function with respect to the truncation of $\SIR$ which is not present in our case. This problem can be overcome using another uniformness, namely in the 

\begin{proof}[Proof of Theorem~\ref{ldpThm}]
As mentioned in the previous paragraph,~\cite[Theorem 1.13]{eiSchm} implies that it suffices to verify the following condition. For every $\varepsilon,K>0$ there exists $b\ge1$ such that 
$$\sup_{\substack{\Q\in\mc{P}_\theta\\ h(\Q|\P)\le K}}\d_{\ms{TV}}(\Q^*,\Q^{*,b})\le\varepsilon.$$
 We show that a slightly stronger statement holds, where $\d_{\ms{TV}}(\Q^*,\Q^{*,b})$ is replaced by $\Q^*(Y^{(o),b}\ne Y^{(o)})$. 

Lemma~\ref{bTruncLem} shows that there exists $b_0\ge1$ such that if $b'>b\ge b_0$, then 
$$\limsup_{n\to\infty}|\L_n|^{-1}\log\P(\#\{X_i\in\L_n:\,Y^{(i),b'}\ne Y^{(i),b}\}>\varepsilon|\L_n|)\le -K.$$
Hence, the LDP from Lemma~\ref{ldpBProp2} yields that 
$$\inf_{\substack{\Q_0\in\mc{P}_\theta\\ \Q^{*,b,b'}_0(Y^{(o),b'}\ne Y^{(o),b})>\varepsilon}}h(\Q|\P)>K.$$
In particular, if $h(\Q|\P)\le K$, then $\Q^{*,b',b}(Y^{(o),b'}\ne Y^{(o),b})\le\varepsilon$, as required.
\end{proof}

\section{Proof of Theorem~\ref{lowTauProbThm}}
\label{sec3}
The difficulty in proving Theorem~\ref{lowTauProbThm} is that the connectable receivers associated with the origin are not stationary, so that we cannot use LDPs for the stationary empirical field directly. Therefore, we first consider a more general stationary marked point process from which the connectable receivers can be reproduced by an application of the contraction principle. Since we need a path-space LDP for this stationary marked point process, we proceed as in the classical proof of Mogulskii's Theorem~\cite[Theorem 5.3.1]{dz98} and use the Dawson-G\"artner technique to deduce the path-space LDP from the finite-dimensional marginals.

In order to define a suitable auxiliary stationary marked point process, we consider the random measure 
$$M^\t(\cdot)=\lambda_{\ms{R}}\int_{\L_\t'}\int_{0}^\infty\one\{\cdot\}\nu_y(\d s)\d y,$$
where $\nu_y([0,s])=\Gamma(s,y)=\E(q(sI(y)^{-1})|X)$ see also \eqref{Gamma}. Then, we let $Z^\t=\{(Y_j,S_j)\}$ denote a Cox process with this random intensity measure and
%Consider the process of receivers $Y=\{Y_j\}_{j\ge1}$ and attach to it iid marks $(S_j)_{j\ge1}$ that are distributed according to $\P(S_j\le s)=q(s)$
%uniformly distributed on $[0,1]$. Then, we 
define the two-parameter field $Y^{*,\t}=\{Y^{*,\t}(x,s)\}_{(x,s)\in\L_1'\times[0,\infty)}$ by 
$$Y^{*,\t}(x,s)=Z^\t(\L_\t(x)\times(0,s]),$$
where 
$$\L_t(\xi^1,\ldots,\xi^d)=t^{-\beta}\prod_{i=1}^d[-|\L_1'|^{1/d}/2,\xi^i].$$
In particular, for any fixed $(x,s)\in\L_1'\times[0,\infty)$, conditioned on $X$ the random variable $Y^{*,\t}(x,s)$ is Poisson-distributed with parameter $\lambda_{\ms{R}}\int_{\L_\t(x)}\Gamma(s,y)\d y$. Moreover, $Y^{*,t}$ is a random variable with values in $L_{\ms{inc}}$, the space of $[0,\infty)$-valued, bounded and coordinate-wise increasing functions on $\L_1'\times[0,\infty)$. 

In the following, we put $\mu_\Q(s)=\lambda_{\ms{R}}\Q(\Gamma(s,o))$ and note that the derivative $\tfrac{\d}{\d s}\mu_\Q(s)$ exists since $q$ is differentiable and $\tfrac{\d}{\d s} q(s)$ is Lipschitz continuous with Lipschitz constant $N$.

Similar to \cite[Section 5.3]{dz98} we introduce the notion of absolute continuity for increasing functions $F: \L'_1\times[0,\infty)\to[0,\infty), (x,s)\mapsto F(x,s)$.
%with respect to the $x$-entry. 
For the convenience of the reader, we reproduce some of these definitions and observations. 
$F$ defines a additive set-function on the set of cubes. More precisely for $\Lambda = (a_1,b_1]\times(a_2,b_2]\times\cdots\times(a_d,b_d]\times(a_{d+1},b_{d+1}]$ we will sometimes write $F(\Lambda):=\sum_u\s(u)F(u)$ with $\s(u):=(-1)^\r$ where $\r=\#\{k:u_k=a_k \}$ and the summation extends of all corners $u$ of $\Lambda$ see~\cite[Chapter 3]{kallenberg}. It follows from Carath\'eodory's extension theorem, that any right-continuous $F\in L_{\ms{Inc}}$ induces a unique measure $\mu_F$ on $\L_1'\times[0,\infty)$ with the Borel sigma-algebra satisfying 
$$\mu_F(\prod_{i=1}^d[-|\L_1'|^{1/d}/2,\xi^i]\times[0,s])=F(\xi^1,\ldots,\xi^d,s),$$
for  any $s\ge0$ and $(\xi^1,\ldots,\xi^d)\in\L_1'$, see~\cite[Theorem 3.25]{kallenberg}.

The function $F$ is called absolutely continuous if $F$ is right-continuous and $\mu_{F}$ is absolutely continuous with respect to the Lebesgue measure on $\L_1'\times[0,\infty)$. We write $\partial F/(\partial x\partial s)$ for its Radon-Nikodym derivative.
%By the Radon-Nikodym theorem when $\partial^{d+1} F/(\partial \xi^1\cdots \partial \xi^d\partial s)\in L^1(\nu_{d+1})$, where $\nu_{d+1}$ denotes the Lebesgue measure on $[-1/2,1/2]^d\times[w',\infty)$, then $d\mu_{F}/d\nu_{d+1}=\partial^{d+1} F/(\partial \xi^1\cdots \partial \xi^d\partial s)=:\partial F^s/(\partial x\partial s)\in L^1(\nu_{d+1})$ ($\nu_{d+1}$ a.e.). 
Let
\begin{equation*}
\begin{split}
AC_{0}^1:=\{F:\,&F \text{ is absolutely continuous, $F(x,0)=0$ and $F(-|\L_1'|^{1/d}/2,\xi^2,\dots,\xi^d,s)=$} \cr
&=F(\xi^1,-|\L_1'|^{1/d}/2,\dots,\xi^d,s)=\cdots=F(\xi^1,\dots,\xi^{d-1},-|\L_1'|^{1/d}/2,s)=0 \}.
\end{split}
\end{equation*}
%\begin{equation*}
%\begin{split}
%AC_0:=\{F: \hspace{0.1cm}&\text{ For Lebesgue almost all } s\in[1,\infty)\hspace{0.1cm} F^s \text{ is absolutely continuous and} \cr
%&F^s(0,\xi^2,\dots,\xi^d)=F^s(\xi^1,0,\dots,\xi^d)=\cdots=F^s(\xi^1,\dots,\xi^{d-1},0)=0,\cr
%\end{split}
%\end{equation*}
In Section~\ref{cpSec}, we will derive Theorem~\ref{lowTauProbThm} by the contraction principle from the following result. 
\begin{proposition}
\label{genLDP}
The random fields $\big\{|\L_\t'|^{-1}Y^{*,\t}(\cdot,\cdot)\}_{\t<1}$ satisfy an LDP in the topology of pointwise convergence with rate $|\L_\t'|$ and good rate function given by 
\begin{align*}
\mc{I}(F)=\begin{cases}
  \int_{\L'_1}\mc{I}^*\big(\frac{\partial F}{\partial y\partial s}(y,\cdot)\big)\d y  & \text{if }F\in AC_{0}^1,\\
  \infty & \text{otherwise,}
\end{cases}
\end{align*}
where
\begin{align}\label{SingleSiteRateFunction2}
\mc{I}^*(g)=
\inf_{\Q\in\mc{P}_\theta} h(\Q|\P)+\int_{0}^\infty h\big(g(s)\big|\tfrac{\d}{\d s}\mu_\Q(s)\big)\d s.
\end{align}
\end{proposition}

%\subsection{Computation of finite-dimensional marginals}

\subsection{Finite-dimensional result}
%In the following, put
%$$M^{\ms{per},\t,Q,r}=|\L_\t'|^{-1}|Q|^{-1}\lambda_{\ms{R}}\int_{|\L_\t'|^{1/d}Q}q(r^{-\alpha}\I_b(y,X^{\ms{per},\t}_o)^{-1})\d y,$$
%and
%$$\mu_{b,\Q,r}=\lambda_{\ms{R}}\Q(q(r^{-\alpha}\I_b(o)^{-1})).$$
In order to apply the Dawson-G\"artner Theorem~\cite[Theorem 4.6.1]{dz98}, we first derive the finite-dimensional LDPs.

\begin{proposition}
\label{finDimLem}
Let $-|\L_1'|^{1/d}/2=\xi_0<\xi_1<\cdots<\xi_k\le|\L_1'|^{1/d}/2$ and $0=s_0< s_1<\cdots<s_r$. Furthermore, put $\Xi=\{\xi_0,\xi_1,\ldots,\xi_k\}$ and $S=\{s_0,s_1,\ldots,s_r\}$. Then, the random variables
$\big\{\big(|\L_\t'|^{-1}Y^{*,\t}(x,s)\big)_{(x,s)\in\Xi^d\times S}\big\}_{\t<1}$ satisfy an LDP with rate $|\L_\t'|$ and good rate function
$$\mc{I}_{\Xi,S}(F)=\sum_{x\in \Xi^d}|\L_\Xi(x)|\inf_{\Q\in\mc{P}_\theta}h(\Q|\P)+\sum_{i=1}^{r}h\Big(\frac{1}{|\L_\Xi(x)|}F(\L_\Xi(x)\times(s_{i-1},s_i])\big|\Delta\mu_{\Q}(s_{i})\Big),$$
where $\Delta\mu_{\Q}(s_{i})=\mu_{\Q}(s_{i})-\mu_{\Q}(s_{i-1})$, and $\L_\Xi(\xi_{i_1},\ldots,\xi_{i_d})=\prod_{j=1}^d(\xi_{i_j},\xi_{i_j+1}]$ with $(\xi_{i_1},\ldots,\xi_{i_d})\in\Xi^d$.
\end{proposition}
The basic idea of proof for Proposition~\ref{finDimLem} is to apply the LDP for the stationary empirical field~\cite[Theorem 3.1]{georgii2}. However, in order to cast our problem into a suitable framework, we first have to perform a truncation and a periodization step.

\subsubsection{Truncation of the path-loss function. }
\label{truncP2Sec}
In a first step, we show that truncation of the path-loss function gives an exponentially good approximation. Let $b\ge1$, $s'\ge s\ge0$ and $x,x'\in\L_1$ be such that all coordinates of $x'-x$ are positive. Then, we let $Y^{*,b,\t}(x,x',s,s')$ denote a random variable that conditioned on the independently marked Poisson particle process $X$ is Poisson distributed with parameter $\lambda_{\ms{R}}\int_{\L_\t({x,x'})}\Gamma^b(s',y)-\Gamma^b(s,y)$, where $\Gamma^b(s,y)=\E(s\I^b(y)^{-1}|X)$, 
$$\L_\t({x,x'})=\t^{-\b}\L(x,x')=\t^{-\b}\prod_{i=1}^d(\pi_k(x),\pi_k(x')],$$%\big(x+\L((-1/2,\ldots,-1/2)+x'-x)\big).$$
and $\pi_k:\R^d\to\R$ denotes the projection onto the $k$th coordinate. From now on let us again assume that the support of the power variables is contained in $[0,s_{\ms{min}}^{-1}]$.
% $s_{\ms{max}}^{-1}\le P_x\le s_{\ms{min}}^{-1}$. 

\begin{lemma}\label{bApprLem}
Let $b\ge1$, $s'\ge s>0$, and $x,x'\in\L'_1$. Then, $\{Y^{*,b,\t}(x,x',s,s')\}_{b\ge1,\t<1}$ are exponentially good approximations of $\{Y^{*,\t}(\L_\t(x,x')\times(s,s'])\}_{\t<1}$.
\end{lemma}

\begin{proof}
Conditioned on $X$, the random variable $|Y^{*,b,\t}(x,x',s,s')-Y^{*,\t}(x,x',s,s')|$ is stochastically dominated by a Poisson distributed random variable with parameter
$$H_b=\lambda_{\ms{R}}\int_{\L_\t(x,x')}\Gamma^b(s',y)-\Gamma(s',y)+\Gamma^b(s,y)-\Gamma(s,y)\d y.$$
Hence, using the Laplace transform of Poisson random variables, for any $a\ge1$ the exponential moment of $a|Y^{*,b,\t}(x,x',s,s')-Y^{*,\t}(x,x',s,s')|$ are bounded from above by 
\begin{align*}
\E\exp(a|Y^{*,b,\t}(x,x',s,s')-Y^{*,\t}(x,x',s,s')|)&\le\E\exp((e^a-1)H_b).
\end{align*}
Now, similar to the proof of Lemma~\ref{bTruncLem}, $H_b$ can be bounded from above by 
$$\lambda_{\ms{R}}N(s+s')s_{\ms{min}}^{-2}w^{-2}\int_{\L_\t(x)}\sum_{i\ge1}\ell(|X_i-y|)-\ell_b(|X_i-y|)\d y,$$ 
so that
\begin{align*}
&\E\exp(a|Y^{*,b,\t}(x,x',s,s')-Y^{*,\t}(\L_\t(x,x')\times(s,s'])|)\\
&\quad\le \E\exp\Big((e^a-1)\lambda_{\ms{R}}N(s+s')s_{\ms{min}}^{-2}w^{-2}\int_{\L_\t(x,x')}\sum_{i\ge1}\ell(|X_i-y|)-\ell_b(|X_i-y|)\d y\Big).
\end{align*}
Now, we conclude as in Lemma~\ref{bTruncLem}.
%$$\E\exp(s(\#Y^{\tau,b}-\#Y^{\tau}))\le\E\exp\Big((e^s-1)\lambda_{\ms{R}}\Big(\k_d\delta^d\tau^{-d\beta}+L\delta^{-\alpha}\int_{\Lambda_\tau'\setminus B_{\delta\tau^{-\beta}}}\I(y)-\I_b(y) \d y\Big)\Big).$$
\end{proof}

\subsubsection{Periodization of the integration domain. } 
Next, we show that replacing the quantity $Y^{*,b,\t}(x,x',s,s')$ by a periodized variant is exponentially equivalent. To be more precise, let $b\ge1$, $s'\ge s\ge0$
and $x,x'\in\L'_1$ be such that all coordinates of $x'-x$ are positive. Then, $X^{\ms{per},\t}$ denotes the periodization of $X\cap\L_\t(x,x')$, i.e., 
$$X^{\ms{per},\t}=\bigcup_{z\in\Z^d}(|\L_\t(x,x')|^{1/d}z+X\cap\L_\t(x,x')).$$
As in Lemma~\ref{bApprLem}, we let $Y^{*,\ms{per},b,\t}(x,x',s,s')$ denote a random variable that conditioned on $X$ is Poisson distributed with parameter 
$$\lambda_{\ms{R}}\int_{\L_\t(x,x')}\Gamma^{\ms{per},b}(s',y)-\Gamma^{\ms{per},b}(s,y)\d y.$$
%q(s'\I^{\ms{per},b}(y)^{-1})-q(s\I^{\ms{per},b}(y)^{-1})\d y.$$ 
Here, $\Gamma^{\ms{per},b}(s,y)=\E(q(s\I^{\ms{per},b}(y)^{-1})|X^{\ms{per},\t})$ and $\I^{\ms{per},b}(y)$ is the interference at $y$ in the periodized configuration computed using truncated path-loss functions.
%Next, we show that the random variables $\big\{Y^{*,\ms{per},b,\t}(x,x',s,s')\big\}_{\t<1}$ and $\big\{Y^{*,b,\t}(x,x',s,s')\big\}_{\t<1}$ are exponentially equivalent.
\begin{lemma}
\label{bApprLem2}
The random variables $\big\{Y^{*,\ms{per},b,\t}(x,x',s,s')\big\}_{\t<1}$ are exponentially equivalent to the random variables $\big\{Y^{*,b,\t}(x,x',s,s')\big\}_{\t<1}$.
\end{lemma}
\begin{proof}
Since we consider truncated interferences, we have that $\I^{\ms{per},b}(y)=\I^{b}(y)$ for all $y\in \L^-_\t(x,x')$, where $\L^-_\t(x,x')$ denotes the subset of all $y\in\L_\t(x,x')$ such that $B_{b}(y)\subset\L_\t(x,x')$. In particular, $|Y^{*,\ms{per},b,\t}(x,x',s,s')-Y^{*,b,\t}(x,x',s,s')|$ is stochastically dominated by a Poisson random variable with parameter $2\lambda_{\ms{R}}|\L_\t(x,x')\setminus\L^-_\t(x,x')|$. Now, we can conclude as in Lemma~\ref{PeriodError} by making use of the Poisson concentration property.
\end{proof}

\subsubsection{Application of LDP for stationary empirical fields. }
We have seen that truncating the interference and considering a periodization does not have an effect on $\{Y^{*,\t}(x,x',s,s')\}_{\t<1}$ in the LDP asymptotics. Now, we derive an LDP after these modifications have been implemented. We put $\mu^{b}_\Q(s)=\Q(\Gamma^{b}(s,o))$.
\begin{proposition}
\label{lowTauProbLem}
The random variables $\big\{|\L_\t'|^{-1}\#Y^{*,\ms{per},b,\t}(x,x',s,s')\}_{\tau<1}$ satisfy an LDP with rate $|\L_\t'|$ and good rate function 
\begin{align}\label{SupInfRep}
%\begin{split}
\mc{I}^{x,x',s,s'}_{b,N}(a)&=|\L(x,x')|\inf_{\Q\in\mc{P}_\theta}\big(h(\Q|\P)+h\big(\tfrac{a}{|\L(x,x')|}|\mu^{b}_\Q(s')-\mu^{b}_\Q(s)\big)\big)
%&=|\L(x,x')|\sup_{u\in\R}\inf_{\Q\in\mc{P}_\theta}h(\Q|\P)+\tfrac{a}{|\Lambda(x,x')|}u-(e^u-1)\big(\mu^{\ms{per},b}_\Q(s')-\mu^{\ms{per},b}_\Q(s)\big).
%\end{split}
\end{align}
%That is, 
%$$\limsup_{\t\to0}\frac{1}{|\L^{\t,N}|}\log\P\Big(\frac{1}{|\L^{\t,N}|}\#Y^{\ms{per},\t,b,N,r}_o\in F\Big)\le-\inf_{x\in F}\mc{I}_{b,N,r}(x)$$
%for all closed $F\subset[0,\infty)$ and
%$$\liminf_{\t\to0}\frac{1}{|\L^{\t,N}|} \log\P\Big(\frac{1}{|\L^{\t,N}|}\#Y^{\ms{per},\t,b,N,r}_o\in G\Big)\ge-\inf_{x\in G}\mc{I}_{b,N,r}(x)$$
%for all open $G\subset[0,\infty)$.
\end{proposition}
%\begin{proof}
%We want to apply ~\cite[Theorem 3.1]{georgii2}. \textbf{We need to relate} the event $|\Lambda^{\t,N}|^{-1}\#Y^{\ms{per},\t,b,N,r}_o\in F$  (expressed in terms of the marked process $Z$) to a function $G(R_{\t,Z})$ where the \textit{stationary empirical fields} $R_{n,Z}$ are defined as $R_{\t,Z}:=|\Lambda^{\t,N}|^{-1}\int_{\Lambda^{\t,N}}\one_{Z-z}dz$.
%\end{proof}
%Let us comment on the relationship between the underlying measurable space versus the underlying topological space. 
Let us recall from \cite[Equations 1.2.12 and 1.2.13]{dz98} that if the random variable considered in an LDP is measurable with respect to the Borel $\s$-algebra on the underlaying topological space, then the proof of the upper and lower bound can be done directly for closed and open sets, respectively. We use this in the sequel without further mentioning.

\medskip
We prepare the proof of Proposition~\ref{lowTauProbLem} by a lemma. First, we note that~\cite[Theorem 3.1]{georgii2} gives the following auxiliary result, where we put 
$$M_{\ms{av},\t}=M_{\ms{av},\t}(x,x',s,s')=\lambda_{\ms{R}}|\L_\t(x,x')|^{-1}\int_{\L_\t(x,x')}\Gamma^{\ms{per},b}(s',y)-\Gamma^{\ms{per},b}(s,y)\d y.$$

\begin{lemma}
\label{coxIntLDP}
Let $F$ and $G$ be compact and open subsets of $[0,\infty)$, respectively. Then,
\begin{align*}
&\limsup_{\t\to0}\frac{1}{|\L_{\t}(x,x')|}\log\E\exp\big(-|\Lambda_{\t}(x,x')|\inf_{a\in F}h(a|M_{\ms{av},\tau})\big)\\
&\quad\le-\inf_{\substack{\Q\in\mc{P}_\theta\\ a\in F}}h(\Q|\P)+h\big(a|\mu^{b}_\Q(s')-\mu^{b}_\Q(s)\big),
\end{align*}
and 
\begin{align*}
&\liminf_{\t\to0}\frac{1}{|\L_{\t}(x,x')|}\log\E\exp\big(-|\Lambda_{\t}(x,x')|\inf_{a\in G}h(a|M_{\ms{av},\tau})\big)\\
&\quad\ge-\inf_{\substack{\Q\in\mc{P}_\theta\\ a\in G}}h\big(\Q|\P)+h(a|\mu^{b}_\Q(s')-\mu^{b}_\Q(s)\big).
\end{align*}
\end{lemma}
\begin{proof}
In order to apply~\cite[Theorem 3.1]{georgii2}, we only need to check that the functions
$$\Q\mapsto\inf_{a\in F}h\big(a|\mu^{b}_\Q(s')-\mu^{b}_\Q(s)\big)$$
and
$$\Q\mapsto\inf_{a\in G}h\big(a|\mu^{b}_\Q(s')-\mu^{b}_\Q(s)\big)$$
are lower- and upper-semicontinuous, respectively. First, note that the map $\Q\mapsto \mu^{b}_\Q(s')-\mu^{b}_\Q(s)$  is continuous in the $\tau_{\mc{L}}$-topology, since $\Gamma^{b}(\cdot,o)$ only depends on $X$ via $X\cap B_b(o)$. Now, we conclude by observing that $a'\mapsto \inf_{a\in F}h(a|a')$ is lower-semicontinuous as pointwise infimum of a two-parameter lower-semicontinuous function over a compact set and $a'\mapsto\inf_{a\in G}h(a|a')$ is upper-semicontinuous as infimum over a family of continuous functions.
\end{proof}

Now, we can proceed with the proof of Proposition~\ref{lowTauProbLem}.
\begin{proof}[Proof of Proposition~\ref{lowTauProbLem}]
The upper bound for compact $F$ is an immediate consequence of Lemma~\ref{coxIntLDP}, since~\cite[Lemma 1.2]{penrose} implies that
\begin{align*}
\P(|\L_\t'|^{-1}Y^{*,\ms{per},b,\t}(x,x',s,s')\in F)&\le\E\exp\Big(-|\L_\t(x,x')|\inf_{a\in F}h\big(\tfrac{a}{|\L(x,x')|}|M_{\ms{av},\tau}\big)\Big).
\end{align*}
The proof of the lower bound is more involved. First, we may assume that $G$ is an interval, i.e., $G=[0,\gamma)$ or $G=(\gamma_-,\gamma_+)$ for some $\gamma,\gamma_-,\gamma_+>0$. Next, introduce the function $f(k)$ by $f(0)=1$ and $f(k)=e^{-1/(12k)}(\sqrt{2\pi k})^{-1}$ for $k\ge1$, and put $G^\t=\Z\cap(|\L_\t'|G)$. Then, by~\cite[Lemma 1.3]{penrose}, 
\begin{align*}
&\P(|\L_\t'|^{-1}Y^{*,\ms{per},b,\t}(x,x',s,s')\in G)\\
&\quad\ge\E\exp\Big(-\inf_{k\in G^\t}-\log f(k)+h\big(k\big||\L_\t(x,x')|M_{\ms{av},\t}\big)\Big)\\
&\quad\ge\E\exp\Big(-|\L_\t'|(|\L_\t'|^{-1/2}+|\L(x,x')|\inf_{k\in G^\t}M_{\ms{av},\t} h\big(k|\L_\t(x,x')|^{-1}{M}_{\ms{av},\t}^{-1})\big)\Big),
%&\ge\exp\Big(-|\Lambda^{\t,N}|(\varepsilon+\inf_{g\in G}\overline{M}\sup_{\sigma\in[-1,1]}H(k|\Lambda^{\t,N}|^{-1}/\overline{M}-\sigma|\Lambda^{\t,N}|^{-1}/\overline{M}))\Big
\end{align*}
where $h(k|\L_{\t}(x,x')|^{-1}{M}_{\ms{av},\t}^{-1})=h(k|\L_\t(x,x')|^{-1}M_{\ms{av},\t}^{-1}\big|1)$ is a short notation.
%\begin{align*}
%&\P(|\Lambda^{\t,N}|^{-1}\#Y_o^{\ms{per},\t,b,N,r}\in G)=\E\P(|\Lambda^{\t,N}|^{-1}\#Y_o^{\ms{per},\t,b,N,r}\in G|M).
%\end{align*}
%For $k\ge0$ write . Moreover, let 
Now, we distinguish between the cases where $G$ contains $0$ and where it does not. We claim that if $G=[0,\gamma)$ and $\varepsilon>0$, then
$$\inf_{k\in G^\t}{M}_{\ms{av},\t} h(k|\L_\t(x,x')|^{-1}{M}_{\ms{av},\t}^{-1})\le\varepsilon+\inf_{g\in G}M_{\ms{av},\t}h(\tfrac{g}{|\L(x,x')|}{M}_{\ms{av},\t}^{-1}),$$
 provided that $\t>0$ is sufficiently small.
Once this claim is proven, Lemma~\ref{coxIntLDP} completes the proof of the lower bound for the case $G=[0,\gamma)$. Let $\varepsilon>0$ be arbitrary. Then, under the event $M_{\ms{av},\t}\le\varepsilon$, we deduce that
$$\inf_{k\in G^\t}M_{\ms{av},\t}h(k|\L_\t(x,x')|^{-1}M^{-1}_{\ms{av},\t})\le M_{\ms{av},\t}\le \varepsilon.$$
%In particular, 
%$$\inf_{k\in G^\t}\overline{M}h(k|\Lambda^{\t,N}|^{-1}/\overline{M})\le \varepsilon$$
%if $\overline{M}\le \varepsilon$. 
On the other hand, if $M_{\ms{av},\t}\ge\varepsilon$, then for every $g\in G$,
$$\big|g|\L(x,x')|^{-1}M_{\ms{av},\t}^{-1}-k(g)|\L_\t(x,x')|^{-1}M_{\ms{av},\t}^{-1}\big|\le|\L_\t(x,x')|^{-1}\varepsilon^{-1},$$
where $k(g)\ge1$ is chosen as the element of $\Z\cap(|\L_{\t}'|G)$ such that $k(g)|\L_{\t}(x,x')|^{-1}$ minimizes the distance to $g|\L(x,x')|^{-1}$. In particular, uniform continuity of $h(\cdot)$ on the interval $[0,\gamma|\L(x,x')|^{-1}\varepsilon^{-1}]$ implies that 
$$\inf_{k\in G^\t}M_{\ms{av},\t}h(k|\L_\t(x,x')|^{-1}M_{\ms{av},\t}^{-1})\le\varepsilon+\inf_{g\in G}M_{\ms{av},\t}h(g|\L(x,x')|^{-1}M_{\ms{av},\t}^{-1}),$$
for all sufficiently small $\t>0$. Finally, we deal with the case, where $G=(\gamma_-,\gamma_+)$ and observe that if $M_{\ms{av},\t}\ge\varepsilon$, then we can conclude as before. To be more precise, 
\begin{align*}
&\E\exp\Big(-|\L_\t(x,x')|\inf_{k\in G^\t}{M}_{\ms{av},\t} h(k|\L_\t(x,x')|^{-1}{M}_{\ms{av},\t}^{-1})\Big)\\
&\quad\ge\E\exp\Big(-|\L_\t(x,x')|(-\varepsilon+\inf_{g\in G}{M}_{\ms{av},\t}h\big(g|\L(x,x')|^{-1}M_{\ms{av},\t}^{-1})\big)\Big)\\
&\quad\phantom{\ge}-\E\one\{M_{\ms{av},\t}\le\varepsilon\}\exp\big(-|\L_\t(x,x')|(-\varepsilon+\inf_{g\in G}{M}_{\ms{av},\t}h\big(g|\L(x,x')|^{-1}M_{\ms{av},\t}^{-1})\big)\big).
\end{align*}
Now, for any $K\ge1$ there exists $\varepsilon>0$ such that  $M_{\ms{av},\t}h(\gamma_-|\L(x,x')|^{-1}M_{\ms{av},\t}^{-1})\ge K$ if $M_{\ms{av,\t}}\le\varepsilon$.
In particular, 
$$\exp\big(-|\L_\t(x,x')|(-\varepsilon+\inf_{g\in G}{M}_{\ms{av},\t}h(g{M}_{\ms{av},\t}^{-1}))\big)\le\exp\big(-|\L_\t(x,x')|(-\varepsilon+K)\big),$$
which completes the proof of the lower bound.

Since $Y^{*,\ms{per},b,\t}(x,x',s,s')$ is stochastically dominated by a Poisson random variable with parameter $\lambda_{\ms{R}}|\L_\t(x,x')|$, the random variables $\big\{|\L_\t'|^{-1}Y^{*,\ms{per},b,\t}(x,x',s,s')\}_{\t<1}$ are exponentially tight. This implies both goodness of the rate function and the full LDP.
%Finally, the second representation in~\eqref{SupInfRep} is a consequence of the min-max theorem, see e.g.~\cite[Exercise 2.2.38]{dz98}. Indeed, the function 
%$$(u,\Q)\mapsto h(\Q|\P)+\tfrac{t}{|\Lambda(x,x')|}u-(e^u-1)(\mu^{\ms{per},b}_\Q(s')-\mu^{\ms{per},b}_\Q(s))$$
%is convex and lower-semicontinuous in $\Q$ and concave and upper-semicontinuous in $u$.
\end{proof}

Next, using Lemma~\ref{lowTauProbLem}, we derive an LDP for the finite-dimensional marginals of $Y^{*,\ms{per},b,\t}(\cdot,\cdot)$. In order to state this precisely, it is convenient to introduce some notation. Let $-|\L_1'|^{1/d}/2=\xi_0<\xi_1<\cdots<\xi_k\le|\L_1'|^{1/d}/2$ and $0=s_0< s_1<\cdots<s_r$. Then, for $x=(\xi_{i_1},\ldots,\xi_{i_d})$ and $s=s_i$ we put $x_{+,\Xi}=(\xi_{i_1+1},\ldots,\xi_{i_d+1})$ and $s_{+,S}=s_{i+1}$, where we use the conventions $\xi_{k+1}=|\L_1'|^{1/d}/2$ and $s_{\ell+1}=\infty$. 
\begin{corollary}
\label{multTauProb2Lem}
Let $-|\L_1'|^{1/d}/2=\xi_0<\xi_1<\cdots<\xi_k\le|\L_1'|^{1/d}/2$ and $0=s_0< s_1<\cdots<s_\ell$. Furthermore, put $\Xi=\{\xi_0,\xi_1,\ldots,\xi_k\}$ and $S=\{s_0,s_1,\ldots,s_r\}$. Then, the random vectors
$$\big\{\big(|\L_\t'|^{-1}Y^{*,b,\t}(x,x_{+,\Xi},s,s_{+,S})\big)_{(x,s)\in\Xi^d\times S}\big\}_{\t<1}$$ 
satisfy an LDP with rate $|\L_\t'|$ and good rate function
$$\mc{I}^{b}_{\Xi,S}((a_z)_{z\in\Xi^d\times S})=\sum_{x\in\Xi^d}|\L(x,x_{+,\Xi})|\inf_{\Q\in\mc{P}_\theta}h(\Q|\P)+\sum_{i=1}^{r} h\big(\tfrac{a_{x,s_i}}{|\L(x,x_{+,\Xi})|}\big|\Delta\mu^{b}_{\Q}(s_{i})\big).$$
%where $\Delta\mu^{b}_{\Q}(s_{i})=\mu^{b}_{\Q}(s_{i})-\mu^{b}_{\Q}(s_{i-1})$.
%Let $-1/2=\xi_0<\xi_1<\cdots<\xi_k\le1/2$ and $w=s_0< s_1<\cdots<s_\ell$. Furthermore, put $\Xi=\{\xi_0,\xi_1,\ldots,\xi_k\}$ and $S=\{s_0,s_1,\ldots,s_\ell\}$. Then, the random vectors
%$$\big\{\big(|\L_\t'|^{-1}Y^{*,\ms{per},b,\t}(x,x_{+,\Xi},s,s_{+,S})\big)_{(x,s)\in\Xi^d\times S}\big\}_{\t<1}$$ 
%satisfy an LDP with rate $|\L_\t'|$ 
%and good rate function
\end{corollary}
\begin{proof}
First, we observe that $\big(Y^{*,\ms{per},b,\t}(x,x_{+,\Xi})\big)_{x\in\Xi^d}$ defines a family of independent random vectors, where we put 
$$Y^{*,\ms{per},b,\t}(x,x_{+,\Xi})=\big(Y^{*,\ms{per},b,\t}(x,x_{+,\Xi},s,s_{+,S})\big)_{s\in S}.$$
Indeed, this is a consequence of the independence property of the Poisson point process $X$ since by the definition of the periodization, $Y^{*,\ms{per},b,\t}(x,x_{+,\Xi})$ depends on $X$ only via $X\cap\L_\t(x,x_{+,\Xi})$. Hence, if for any fixed $x\in\Xi^d$ we can establish an LDP for $Y^{*,\ms{per},b,\t}(x,x_{+,\Xi})$ with a certain good rate function, then~\cite[Exercise 4.2.7]{dz98} allows us to deduce that the collection $\big(Y^{*,\ms{per},b,\t}(x,x_{+,\Xi})\big)_{x\in\Xi^d}$ satisfies an LDP with good rate function given by the sum of the individual ones. If $|S|=1$, then the LDP for $Y^{*,\ms{per},b,\t}(x,x_{+,\Xi})$ is precisely the result of Lemma~\ref{lowTauProbLem}, and an inspection of its proof shows that it also extends to the case of general finite $S$.

In Lemma~\ref{bApprLem} we have seen that periodization replaces $\{Y^{*,b,\t}(x,x',s,s')\}_{\tau<1}$ by exponentially equivalent random variables. Hence, applying~\cite[Theorem 4.2.13]{dz98} completes the proof.
\end{proof}

In order to deduce Proposition~\ref{finDimLem} from Corollary~\ref{multTauProb2Lem}, we need to undo the truncation approximations. Before we start with the proof of Proposition~\ref{finDimLem}, it is convenient to derive certain continuity properties of $\mu_{\Q}^b$ and $\tfrac{\d}{\d s}\mu_\Q^b(s)$ with respect to $b$ and $\Q$. The technique of proof is similar to the one used in the proof of Theorem~\ref{ldpThm}.

\begin{lemma}
\label{qContLem}
Let $\varepsilon,K>0$ and $s\ge0$ be arbitrary. Then, there exists $b\ge1$ such that if $\Q\in\mc{P}_\theta$ satisfies $h(\Q|\P)\le K$, then
$$\mu^b_\Q(s)-\mu_\Q(s)\le\varepsilon,\qquad\text{and}\qquad |\tfrac{\d}{\d s}\mu_\Q^b(s)-\tfrac{\d}{\d s}\mu_\Q(s)|\le\varepsilon.$$
\end{lemma}
\begin{proof}
We first deal with the part of the statement not involving derivatives. Let $b'\ge b\ge1$ be arbitrary. Proceeding as in Lemma~\ref{coxIntLDP}, we see that $\big\{|\L_\t'|^{-1}\int_{\L_\t'}
\Gamma^b(s,y)-\Gamma^{b'}(s,y)\d y\big\}_{\t<1}$ satisfies an LDP with rate $|\L_\t'|$ and good rate function
$$a\mapsto\inf_{\substack{\Q\in\mc{P}_\theta\\ \mu^b_\Q(s)-\mu^{b'}_\Q(s)=a}}h(\Q|\P).$$
In particular, the proof of the lemma is completed, once we show the existence of $b_0\ge1$ such that
$$\limsup_{\t\to0}|\L_\t'|^{-1}\log\P\Big(\int_{\L_\t'}|\L_\t'|^{-1}\Gamma^b(s,y)-\Gamma^{b'}(s,y)\d y>\varepsilon\Big)\le -K.$$
for all $b'\ge b\ge b_0$.
Note that 
\begin{equation}\label{qEstimate}
\begin{split}
\Gamma^b(s,y)-\Gamma^{b'}(s,y)\le Nss_{\ms{min}}^{-2}w^{-2}\sum_{i\ge1}\ell_{b'}(|y-X_i|)-\ell_b(|y-X_i|).
\end{split}
\end{equation}
Hence, using the formula for the Laplace functional of a Poisson point process shows that for any $a>0$, 
\begin{align*}
&\E\exp\Big(a\int_{\L_\t'}\Gamma^b(s,y)-\Gamma^{b'}(s,y)\d y>\varepsilon\Big)\\
&\quad\le\exp\Big(\lambda_{\ms{T}}\int_{\R^d}\exp\big(aNs_{\ms{min}}^{-2}w^{-2}\int_{\L_\t'}\ell(|x-y|)-\ell_b(|x-y|)\d y\big)-1\d x\Big).
\end{align*}
 Now, we conclude as in Lemma~\ref{bTruncLem}.

For the part involving the derivatives, note that the derivative of $\mu_\Q(s)$ is given by \linebreak $\lambda_{\ms{R}}\Q(\I(o)^{-1}\tfrac{\d}{\d s} \Gamma(s,o))$. Essentially, this means replacing in the above arguments the expression $q(s\I(y)^{-1})$ by $\I(y)^{-1}\tfrac{\d}{\d s} q(s\I(y)^{-1})$. This specific form only comes into play in the estimate~\eqref{qEstimate} which can be replaced by 
\begin{equation*}
\begin{split}
&|\I^b(y)^{-1}\tfrac{\d}{\d s} q(s\I^b(y)^{-1})-\I^{b'}(y)^{-1}\tfrac{\d}{\d s} q(s\I^{b'}(y)^{-1})|\\
&\quad\le \I^b(y)^{-1}|\tfrac{\d}{\d s} q(s\I^b(y)^{-1})-\tfrac{\d}{\d s} q(s\I^{b'}(y)^{-1})|+\tfrac{\d}{\d s} q(s\I^{b'}(y)^{-1})|\I^b(y)^{-1}-\I^{b'}(y)^{-1}|\\
&\quad\le \big(Nss_{\ms{min}}^{-2}w^{-3}+Nw^{-2}s_{\ms{min}}^{-2}\big)\sum_{i\ge1}(\ell_{b'}(|y-X_i|)-\ell_{b}(|y-X_i|)),
\end{split}
\end{equation*}
as required.
\end{proof}

\begin{corollary}
\label{qContCor}
Let $K>0$ and $s\ge1$ be arbitrary. Then, in the $\tau_{\mc{L}}$-topology,
\begin{enumerate}
\item as $b\to\infty$, the functions $\Q\mapsto \mu^b_\Q(s)$ converge to $\mu_\Q(s)$ uniformly in $\{\Q:\,h(\Q|\P)\le K\}$. In particular, $\Q\mapsto\mu_\Q(s)$ is continuous on $\{\Q:\,h(\Q|\P)\le K\}$,
\item as $b\to\infty$, the functions $\Q\mapsto \tfrac{\d}{\d s}\mu^b_\Q(s)$ converge to $\tfrac{\d}{\d s}\mu_\Q(s)$ uniformly in $\{\Q:\,h(\Q|\P)\le K\}$. In particular, $\Q\mapsto\tfrac{\d}{\d s}\mu_\Q(s)$ is continuous on $\{\Q:\,h(\Q|\P)\le K\}$,
%\item the function $\Q\mapsto\tfrac{\d}{\d s}\mu_\Q(s)$ is continuous on $\{\Q:\,h(\Q|\P)\le K\}$,
\item if $b_n\to \infty$, $\Q_n\to \Q$ and $\limsup_{n\to\infty}h(\Q_n|\P)\le K$, then $\mu^{b_n}_{\Q_n}(s)\to\mu_\Q(s)$.
\end{enumerate}
\end{corollary}
\begin{proof}
Since the first two items are an immediate consequence of Lemma~\ref{qContLem}, we only deal with the last item.
Here, we use the decomposition
$$|\mu_{\Q}(s)-\mu^{b_n}_{\Q_n}(s)|\le |\mu_{\Q}(s)-\mu_{\Q}^b(s)|+|\mu_{\Q}^b(s)-\mu_{\Q_n}^b(s)|+|\mu^{b_n}_{\Q_n}(s)-\mu_{\Q_n}^b(s)|,$$
and conclude as before.
\end{proof}
%For later use we also prove the preceding result for different observables, namely 
Now, we have completed all preparations for the proof of Proposition~\ref{finDimLem}.

\begin{proof}[Proof of Proposition~\ref{finDimLem}]
In Lemma~\ref{bApprLem}, we have seen that truncation leads to an exponentially good approximation. Therefore, combining Corollary~\ref{multTauProb2Lem} with~\cite[Theorem 4.2.16]{dz98} shows that the random vector $\big\{\big(|\L_\t'|^{-1}Y^{*,\t}(x,s)\big)_{(x,s)\in\Xi^d\times S}\big\}_{\t<1}$ satisfies a weak LDP with rate function
$$\mc{I}_{\Xi,S}'((a_{x,s})_{(x,s)})=\sup_{m\ge1}\liminf_{b\to\infty}\inf_{(a'_{x,s})_{(x,s)}:\,|(a'_{x,s})_{(x,s)}-(a_{x,s})_{(x,s)}|_\infty\le m^{-1}}\mc{I}^{b}_{\Xi,S}((a'_{x,s})_{(x,s)}).$$
Since the random vectors $\big\{\big(|\L_\t'|^{-1}Y^{*,\t}(x,s)\big)_{(x,s)\in\Xi^d\times S}\big\}_{\tau<1}$ are exponentially tight, the proof is completed once we show that $\mc{I}_{\Xi,S}'((a_{x,s})_{(x,s)})=\mc{I}_{\Xi,S}((a_{x,s})_{(x,s)})$, where
$$\mc{I}_{\Xi,S}((a_{x,s})_{(x,s)})=\sum_{x\in \Xi^d}|\L(x,x_{+,\Xi})|\inf_{\Q\in\mc{P}_\theta}h(\Q|\P)+\sum_{i=1}^rh\Big(\tfrac{a_{x,s_{i}}}{|\L(x,x_{+,\Xi})|}\big|\Delta\mu_{\Q}(s_i)\Big).$$
First, we note that $q(s\I^b(o)^{-1})$ is decreasing in $b$ and converges to $q(s\I(o)^{-1})$ as $b\to\infty$. Hence, for any $\Q\in\mc{P}_\theta$ and $x\in\Xi^d$, 
$$\lim_{b\to\infty}\sum_{i=1}^r h\big(\tfrac{a_{x,s_i}}{|\L(x,x_{+,\Xi})|}\big|\Delta\mu^b_\Q(s_i)\big)=\sum_{i=1}^r h\big(\tfrac{a_{x,s_i}}{|\L(x,x_{+,\Xi})|}\big|\Delta\mu_\Q(s_i)\big).$$
In particular, $\mc{I}_{\Xi,S}'((a_{x,s})_{(x,s)})\le\mc{I}_{\Xi,S}((a_{x,s})_{(x,s)})$. For the other direction, fix $\delta>0$ and $x\in\Xi^d$. Then, for each $m\ge1$ choose a sequence $(b_{m,n})_{n\ge1}$ such that $\lim_{n\to\infty}b_{m,n}=\infty$ and write
\begin{align*}
&\lim_{n\to\infty}\hspace{-0.5cm}\inf_{\substack{\Q\in\mc{P}_\theta\\ (a'_{x,s})_{(x,s)}:\,|(a'_{x,s})_{s}-(a_{x,s})_{s}|_\infty\le m^{-1}}}\hspace{-0.5cm}h(\Q|\P)+\sum_{i=1}^rh\Big(\tfrac{a'_{x,s_{i}}}{|\L(x,x_{+,\Xi})|}\big|\Delta \mu^{b_{m,n}}_{\Q}(s_i)\Big)\\
&\quad=\liminf_{b\to\infty}\hspace{-0.5cm}\inf_{\substack{\Q\in\mc{P}_\theta\\ (a'_{x,s})_{(x,s)}:\,|(a'_{x,s})_{s}-(a_{x,s})_{s}|_\infty\le m^{-1}}}\hspace{-0.5cm}h(\Q|\P)+\sum_{i=1}^rh\Big(\tfrac{a'_{x,s_{i}}}{|\L(x,x_{+,\Xi})|^{}}\big|\Delta\mu^b_{\Q}(s_i)\Big).
\end{align*}
Next, for each $m,n\ge1$ choose $\Q_{x,m,n}\in\mc{P}_\theta$, and for each $m,n\ge1$ and $s\in S$ choose $a'_{x,s,m,n}\in [a_{x,s}-1/m,a_{x,s}+1/m]$ such that 
\begin{align*}
&\inf_{\substack{\Q^*\in\mc{P}_\theta\\ (a^*_{x,s})_{s}:\,|(a^*_{x,s})_{s}-(a_{x,s})_{s}|_\infty\le m^{-1}}}h(\Q^*|\P)+\sum_{i=1}^r h\Big(\tfrac{a^*_{x,s_i}}{|\L(x,x_{+,\Xi})|}\Big|\Delta\mu^{b_{m,n}}_{\Q^*}(s_i)\Big)\\
&\quad\ge-\delta+h(\Q_{x,m,n}|\P)+\sum_{i=1}^r h\Big(\tfrac{a'_{x,s_i,m,n}}{|\L(x,x_{+,\Xi})|}\big|\Delta\mu^{b_{m,n}}_{\Q_{x,m,n}}(s_i)\Big).
\end{align*}
If $\limsup_{n\to\infty} h(\Q_{x,m,n}|\P)=\infty$, then 
$$\lim_{n\to\infty}\hspace{-0.6cm}\inf_{\substack{\Q\in\mc{P}_\theta\\ (a'_{x,s})_{(x,s)}:\,|(a'_{x,s})_{s}-(a_{x,s})_{s}|_\infty\le m^{-1}}}\hspace{-0.6cm}h(\Q|\P)+\sum_{i=1}^rh\Big(\tfrac{a'_{x,s_{i}}}{|\L(x,x_{+,\Xi})|}\big|\Delta\mu^{b_{m,n}}_{\Q}(s_i)\Big)=\infty,$$
which is certainly at least as large as $\mc{I}_{\Xi,S}((a_{x,s})_{(x,s)})$. Otherwise, after passing to a subsequence, we may assume that $\lim_{n\to\infty}\Q_{x,m,n}=\Q_{x,m}$ for some $\Q_{x,m}\in\mc{P}_\theta$ by sequential compactness. Furthermore, we may also assume for each $1\le i\le r$, that $\lim_{n\to\infty}a'_{x,s_i,m,n}=a'_{x,s_i,m}$ for some $a'_{x,s_i,m}\in[a_{x,s}-1/m,a_{x,s}+1/m]$. In particular, lower semicontinuity of $h$ implies that $\liminf_{n\to\infty}h(\Q_{x,m,n}|\P)\ge h(\Q_{x,m}|\P)$. Moreover, by Corollary~\ref{qContCor},
$$\lim_{n\to\infty}\Delta\mu^{b_{m,n}}_{\Q_{x,m,n}}(s_i)=\Delta\mu_{\Q_{x,m}}(s_i).$$
Hence, another application of lower semicontinuity gives that
$$\liminf_{n\to\infty}h\Big(\tfrac{a'_{x,s_i,m,n}}{|\L(x,x_{+,\Xi})|}\Big|\Delta\mu^{b_{m,n}}_{\Q_{x,m,n}}(s_i)\Big)\ge h\Big(\tfrac{a'_{x,s_i,m}}{|\L(x,x_{+,\Xi})|}\Big|\Delta\mu_{\Q_{x,m}}(s_i)\Big).$$
Arguing as above, we may assume that $\Q_{x,m}$ converges to some $\Q_x$ as $m\to\infty$. 
In order to conclude the proof of the proposition, it therefore suffices to show that
$$\liminf_{m\to\infty}h\Big(\tfrac{a'_{x,s_i,m}}{|\L(x,x_{+,\Xi})|}\Big|\Delta\mu_{\Q_{x,m}}(s_i)\Big)\ge h\Big(\tfrac{a_{x,s_i}}{|\L(x,x_{+,\Xi})|}\Big|\Delta\mu_{\Q_x}(s_i)\Big).$$
A final application of lower semicontinuity completes the proof.
\end{proof}

\subsection{Application of Dawson-G\"artner \& identification of rate function}
In Proposition~\ref{finDimLem}, we have shown that the finite-dimensional distributions of the random fields $\{Y^{*,\t}\}_{\t<1}$ satisfy an LDP and we have also identified the good rate function. Hence, the Dawson-G\"artner Theorem~\cite[Theorem 4.6.1]{dz98} implies that the random fields $\{Y^{*,\t}\}_{\t<1}$ satisfy an LDP with respect to the topology of pointwise convergence and that the good rate function is given by 
$$\tilde{\mc{I}}(F)=\sup_{\Xi,S}\mc{I}_{\Xi,S}(F),$$
where the supremum is over all finite $S\subset[0,\infty)$ and $\Xi\subset[-|\L_1'|^{1/d}/2,|\L_1'|^{1/d}/2]$. The proof of Proposition \ref{genLDP} now amounts  to showing $\tilde{\mc{I}}(F)=\mc{I}(F)$.
%, we present a more concise representation of $\widetilde{\mc{I}}$. 
This can be done using an adaptation of arguments appearing in the classical derivation of Mogulskii's Theorem provided in~\cite[Theorem 5.3.1]{dz98}. For the convenience of the reader, we provide some details.
%\begin{lemma}
%\label{iRepLem}
%%$$\tilde\mc{I}(F)=\int_{\L_1}\mc{I}_y\big(\frac{\partial F}{\partial y\partial s}\big)\d y.$$
%$$\tilde{\mc{I}}(F)=\mc{I}(F).$$
%\end{lemma}
\begin{proof}[Proof of Proposition \ref{genLDP}]
First assume that $F\in AC_0^1$ and let $f=\partial F/(\partial x\partial s)$ denote the density of $F$. By non-negativity of $\mc{I}_{\Xi,S}(F)$, we can assume $\xi_k=\L_1'|^{1/d}/2$.
% and $\lim_{k}s_k=\infty$. 
Note that $\mc{I}_{\Xi,S}(F)$ is of the form $\sum_{x\in \Xi^d}|\Lambda_\Xi(x)|\inf_{\Q}f(F(x),\Q)$ where $f$ is convex in the pair $(F(x),\Q)$ by linearity of $\mu_\Q$. Hence, also $G(F(x))=\inf_{\Q}f(F(x),\Q)$ is convex in $F(x)$ so that Jensen's inequality gives that
$$\mc{I}_{\Xi,S}(F)\leq \int_{\L'_1}\inf_{\Q\in\mc{P}_\theta}h(\Q|\P)+\sum_{i=1}^kh\Big(\int_{s_{i-1}}^{s_i}f(x,s)\d s\big|\Delta\mu_{\Q}(s_i)\Big)\d x.$$
Again by convexity of $h$ and an application of Jensen's inequality, we can further estimate
\begin{equation*}
\begin{split}
h\Big(\int_{s_{i-1}}^{s_i}f(x,s)\d s\big|\Delta\mu_{\Q}(s_i)\Big)
%&h\Big(\Delta^{\D x}F^{s_i}(\vec{t})-\D^{\D x}F^{s_{i-1}}(\vec{t})\big|\mu_{\Q}(s_i)-\mu_{\Q}(s_{i-1})\Big)\cr
%&=(s_i-s_{i-1})h\Big(\frac{1}{s_i-s_{i-1}}\int_{s_{i-1}}^{s_i}\Delta^{\D x}\dot F^{s}(\vec{t})\d s\big|\int_{s_{i-1}}^{s_i}\tfrac{\d}{\d s}\mu_{\Q}(s)\d s\Big)\cr
%&\le\int_{s_{i-1}}^{s_i}h\Big(\Delta^{\D x}\dot F^{s}(\vec{t})\big|\tfrac{\d}{\d s}\mu_{\Q}(s)\Big)\d s\cr
&\le\int_{s_{i-1}}^{s_i}h\Big(f(x,s)\big|\tfrac{\d}{\d s}\mu_{\Q}(s)\Big)\d s.\cr
%&=\frac{1}{|\D x|}\int_{s_{i-1}}^{s_i}h\Big(\int_{\D x}\frac{\partial^2F^{s}(\vec{t})}{\partial \vec{t}\partial s}\d \vec{t}\big|\tfrac{\d}{\d s}\mu_{\Q}(s)\int_{\D x}\d \vec{t}\Big)\d s\cr
%&\le\frac{1}{|\D x|}\int_{\D x}\int_{s_{i-1}}^{s_i}h\Big(\frac{\partial^2F^{s}(\vec{t})}{\partial \vec{t}\partial s}\big|\tfrac{\d}{\d s}\mu_{\Q}(s)\Big)\d s\d \vec{t}\cr
\end{split}
\end{equation*}
%if $\frac{\partial F}{\partial y\partial s}$ exists. If $\frac{\partial F}{\partial y\partial s}$ does not exists, there is nothing to show. 
This proves $\tilde{\mc{I}}(F)\le\mc{I}(F)$. 
For the other direction we first consider the supremum over partitions $\Xi$ and let some $S$-partition be fixed. The idea is to use a volume partition into equal sub-cubes with side length going to zero as a lower bound. More precisely, let $\{\r_k(l)\}_{l=1}^{k^d}$ denote the disjoint partition of $\L_1'$ into cubes of volume $|\L_1'|/k^d$ and equal side length $\delta(k)=|\L_1'|^{1/d}/k$. Then,
\begin{equation*}
\begin{split}
\mc{I}_{\Xi,S}(F)&\ge\liminf_{k\to\infty}\sum_{l=1}^{k^d}\frac{1}{k^d}\inf_{\Q\in\mc{P}_\theta}h(\Q|\P)+\sum_{i=1}^rh\big(F(\r_k(l)\times (s_{i-1},s_i])\big|\Delta\mu_{\Q}(s_i)\big)\cr
&=\liminf_{k\to\infty}\int_{\L'_1}\inf_{\Q\in\mc{P}_\theta}h(\Q|\P)+\sum_{i=1}^rh\big(f^k_i(x)\big|\Delta\mu_{\Q}(s_i)\big)\d x\cr
\end{split}
\end{equation*}
where each $f^k_i(x)$ is constant on each of the cubes $\r_k(l)$, $l = 1,\dots,k^d$. By Lebesgue's theorem $f^k_i(x)\to\int_{s_{i-1}}^{s_i}f(x,s)\d s$ for Lebesgue almost all $x$ as $k$ tends to infinity. Therefore, by Fatou's lemma and the lower semicontinuity of the rate function
\begin{equation*}
\begin{split}
\sup_\Xi\mc{I}_{\Xi,S}(F)&\ge\liminf_{k\to\infty}\int_{\L'_1}\inf_{\Q\in\mc{P}_\theta}h(\Q|\P)+\sum_{i=1}^rh\big(f^k_i(x)\big|\Delta\mu_{\Q}(s_i)\big)\d x\cr
&\ge\int_{\L'_1}\liminf_{k\to\infty}\inf_{\Q\in\mc{P}_\theta}h(\Q|\P)+\sum_{i=1}^rh\big(f^k_i(x)\big|\Delta\mu_{\Q}(s_i)\big)\d x\cr
&=\int_{\L'_1}\inf_{\Q\in\mc{P}_\theta}h(\Q|\P)+\sum_{i=1}^rh\big(\int_{s_{i-1}}^{s_i}f(x,s)\d s\big|\Delta\mu_{\Q}(s_i)\big)\d x=\mc{I}_{S}(F).\cr
\end{split}
\end{equation*}
For the supremum over $S$-partitions we use the same approach and consider a partition of intervals $[0,k]$ for $k\in\N$ with constant mesh size $1/k$. Using Fatou's lemma, we have  
\begin{equation*}
\begin{split}
\sup_S\mc{I}_{S}(F)&\ge\int_{\L'_1}\liminf_{k\to\infty}\inf_{\Q\in\mc{P}_\theta}h(\Q|\P)+\sum_{i=1}^{k^2}\frac{1}{k}h\Big(\frac{\int_{(i-1)/k}^{i/k}f(x,s)\d s}{1/k}\Big|\frac{\Delta\mu_{\Q}(i/k)}{1/k}\Big)\d x,\cr
\end{split}
\end{equation*}
where we can look at $k\int_{(i-1)/k}^{{i/k}}f(x,s)\d s$ as a stepfunction $f_x^k$ on $[0,k]$. Similarly, for $k[\mu_{\Q}(i/k)-\mu_{\Q}((i-1)/k)]=k\int_{(i-1)/k}^{i/k}\tfrac{\d}{\d s}\mu_\Q(s)\d s$ with $g_\Q^k$ on $[0,k]$, we have that
\begin{equation*}
\begin{split}
\sum_{i=1}^{k^2}\frac{1}{k}h\Big(k\int_{s_{i-1}}^{s_i}f(x,s)\d s\big| k\Delta\mu_{\Q}(i/k)\Big)=\int_{0}^{k}h\big(f_x^k(r)| g_\Q^k(r)\big)\d r.\cr
\end{split}
\end{equation*}
Now fix $x\in\L_1$ and $k\ge1$ and let $\Q_k^x$ such that 
\begin{equation*}
\begin{split}
&\inf_{\Q\in\mc{P}_\theta}h(\Q|\P)+\int_{0}^{k}h(f_x^k(r)| f_\Q^k(r))\d r= h(\Q_k^x|\P)+\int_{0}^{k}h\big(f_x^k(r)| g_{\Q_k^x}^k(r)\big)\d r,
\end{split}
\end{equation*}
which exists since lower-semicontinuous functions assume their minimum on compact sets. Let $k_n$ be the subsequence such that the limit inferior becomes a limit and for simplicity write again $k$.
Further we can assume $\sup_kh(\Q_k^x|\P)<\infty$ for Lebesgue almost all $x$ since otherwise there is nothing to show. Since $h(\cdot|\P)$ has sequentially compact level sets there exists a cluster point $\Q^x_*$ of $(\Q_k^x)_{k\in\N}$ and by lower semicontinuity 
and Fatou's lemma, we have 
\begin{equation*}
\begin{split}
\sup_S\mc{I}_{S}(F)&\ge\int_{\L'_1}h(\Q_*^x|\P)+\int_{0}^{\infty}\liminf_{k\to\infty}h\big(f_x^k(r)\big| g_{\Q_k^x}^k(r)\big)\d r\d x.
\end{split}
\end{equation*}
Note that by Lebesgue's theorem for almost all $s\in[0,\infty)$, $\liminf_{k\to\infty}f_x^k(s)=f(x,s)$. Further note that also 
$\liminf_{k\to\infty}g_{\Q_k^x}^k(s)=\tfrac{\d}{\d s}\mu_{\Q^x_*}(s)$. Indeed by the mean value theorem for $s\in((i-1)/k,i/k)$ there exists $s'\in((i-1)/k,i/k)$ such that $g_{\Q_k^x}^k(s)=\tfrac{\d}{\d s}\mu_{\Q^x_k}(s')$ and
\begin{equation*}
\begin{split}
|g_{\Q_k^x}^k(s)-\tfrac{\d}{\d s}\mu_{\Q^x_*}(s)|\le|\tfrac{\d}{\d s}\mu_{\Q^x_k}(s')-\tfrac{\d}{\d s}\mu_{\Q^x_k}(s)|+|\tfrac{\d}{\d s}\mu_{\Q^x_k}(s)-\tfrac{\d}{\d s}\mu_{\Q^x_*}(s)|.
\end{split}
\end{equation*}
The second summand on the right tends to zero as $k$ tends to infinity by Corollary~\ref{qContCor}. For the first term we have by Lebesgue's theorem
\begin{equation*}
\begin{split}
|\tfrac{\d}{\d s}\mu_{\Q^x_k}(s')-\tfrac{\d}{\d s}\mu_{\Q^x_k}(s)|&\le\lambda_{\ms{R}}\Q^x_k(\E(\tfrac{\d}{\d s}_{|_{s=s'}}q(sI(o)^{-1})-\tfrac{\d}{\d s}_{|_{s=s}}q(sI(o)^{-1})|X))\cr
&\le Nw^{-2}|s'-s|
\end{split}
\end{equation*}
which tends to zero as $k$ tends to infinity and thus $\liminf_{k\to\infty}g_{\Q_k^x}^k(s)=\tfrac{\d}{\d s}\mu_{\Q^x_*}(s)$. 
Using this and lower semicontinuity gives 
% and note that $\tfrac{\d}{\d s}\mu_{\Q^x_*}(s)<\la_{\mc{R}}L$ where $L$ is the Lipschitz constant of $q$.
%and $\inf_{k}f_x^k(s)/f_{\Q_k^x}^k(s)>0$, then 
 $$\liminf_{k\to\infty}h(f_x^k(s)| g_{\Q_k^x}^k(s))\ge h(f(x,s)\big|\tfrac{\d}{\d s}\mu_{\Q^x_*}(s)),$$
as required.
%and thus letting $\de$ go to zero 
%and applying another infimum 
%leads to $\sup_S\mc{I}_{S}(F)\ge\mc{I}(F)$.

\medskip
Finally let $F\notin AC_0$. 
%We have to show $\tilde{\mc{I}}(F)=\infty$. 
First, for any $\e>0$ there exists $\Q_{S,x}\in\mc{P}_\theta$ such that 
\begin{equation*}
\begin{split}
\tilde{\mc{I}}(F)&\ge\sup_{\Xi,S}\Big[\sum_{x\in \Xi^d}|\L_\Xi(x)|h(\Q_{S,x}|\P)\\
&\phantom{\ge}+\sum_{i=1}^r|\L_{\Xi}(x)|h\Big(\tfrac{F(\L_{\Xi}(x)\times(s_{i-1},s_i])}{|\L_{\Xi}(x)|}\big|\Delta\mu_{\Q_{S,x}}(s_i)\Big)\Big]-\e\cr
&\ge\sup_{\Xi,S}\Big[\sum_{x\in \Xi^d}\sum_{i=1}^r|\L_{\Xi}(x)|h\Big(\tfrac{F(\L_{\Xi}(x)\times(s_{i-1},s_i])}{|\L_{\Xi}(x)|}\big|\Delta\mu_{\Q_{S,x}}(s_i)\Big)\Big]-\e\cr
&=\sup_{\Xi,S}\Big[\sum_{x\in \Xi^d}\sum_{i=1}^r|\L_{\Xi}(x)|\sup_\r[\r\tfrac{F(\L_{\Xi}(x)\times(s_{i-1},s_i])}{|\L_{\Xi}(x)|}-(e^\r-1)\Delta\mu_{\Q_{S,x}}(s_i)]\Big]-\e\cr
\end{split}
\end{equation*}
using also the Legendre transform of the relative entropy. Further, we have 
\begin{equation*}
\begin{split}
|\mu_{\Q_{S,x}}(s_i)-\mu_{\Q_{S,x}}(s_{i-1})|&\le\lambda_{\ms{R}}|\Q_{S,x}(\Gamma(s_i,o)-\Gamma(s_{i-1},o)|\le N\lambda_{\ms{R}}w^{-2}|s_i-s_{i-1}|,
\end{split}
\end{equation*}
and hence for $\r\ge0$
\begin{align*}
%\begin{split}
\tilde{\mc{I}}(F)
%\r\sup_{\Xi,S}\big[\sum_{x\in \Xi^d}\sum_{i=1}^r|\Delta x|[\Delta^{\D x}F^{s_i}(x)-\D^{\D x}F^{s_{i-1}}(x)-(e^\r-1)L\lambda_{\ms{R}}|s_i-s_{i-1}|]\big]-\e\cr
&\ge\r\sup_{\Xi,S}\big[\sum_{x\in \Xi^d}\sum_{i=1}^r[F(\L_{\Xi}(x)\times(s_{i-1},s_i])-(e^\r-1)N\lambda_{\ms{R}}|\L_{\Xi}(x)||s_i-s_{i-1}|]\big]-\e.
%\end{split}
\end{align*}
If $F$ is not right-continuous, there exists a point $(x,s)$ such that $F(x,s)<\lim_{n\to\infty}F(x+1/n,s+1/n)=M$. Consider a sequence of finite partitions $(\Xi_n,S_n)_{n\in\N}$ where the cube $(\prod_{j=1}^d(x_j,x_j+1/n])\times(s,s+1/n]$ is contained in $(\Xi_n,S_n)$ for all $n\in\N$. Then
\begin{align*}
\begin{split}
\tilde{\mc{I}}(F)&\ge\r\big[F(x+1/n,s+1/n)-F(x,s)-(e^\r-1)N\lambda_{\ms{R}}1/n^{d+1}\big]-\e
\end{split}
\end{align*}
and letting $n$ tend to infinity gives $\tilde{\mc{I}}(F)\ge\r[M-F(x,s)]-\e$ which tends to infinity for $\r\to\infty$.

If $F$ is right-continuous but $F\notin AC_0$ there exists $\de>0$ and a sequence of measurable sets $A_k$, with $\nu_{d+1}(A_k)\to0$ and $\mu_F(A_k)\ge\de$. Using the regularity of the Lebesgue measure there exists a disjoint union of countably many $d+1$-dimensional cuboids such that $A_k\subset\bigcup_lq^k_l$ and 
$\nu_{d+1}(\bigcup_lq^k_l\setminus A_k)<\frac{1}{k}$. Then, for every $\r\ge0$,
\begin{equation*}
\begin{split}
\tilde{\mc{I}}(F)&\ge\r\sum_{l=1}^\infty F(q^k_l)-(e^\r-1)N\lambda_{\ms{R}}\nu_{d+1}(\bigcup_lq^k_l)-\e\cr
&\ge\r\mu_F(A_k)-(e^\r-1)N\lambda_{\ms{R}}(\nu_{d+1}(A_k)+1/k)-\e.\cr
\end{split}
\end{equation*}
Letting $k$ tend to infinity we have $\tilde{\mc{I}}(F)\ge\r\de-\e$ which tends to infinity as $\r$ tends to infinity.
\end{proof}

\subsection{Contraction principle \& identification of rate function}
\label{cpSec}
In the present section, we apply the contraction principle to derive Theorem~\ref{lowTauProbThm} from Proposition~\ref{genLDP}. Consider the function $\Psi:\mc{K}\to\mc{M}(\L'_1)$ given by 
$$F(\cdot)\mapsto F\Big((\cdot\times [0,\infty))\cap\{(y,s)\in\L'_1\times[0,\infty):s\le|y|^{-\alpha}\}\Big).$$
%Then, $\Psi(Y^{*,\t})$ is just the point process of receivers in $\L_\t'$ that can connect to the origin. 
Then, the random measure $\Psi(|\L_\t'|^{-1}Y^{*,\t}(\cdot,\cdot))$ is exponentially equivalent to the random measure $|\L_\t'|^{-1}Y^{\t}(t^{-\beta}\cdot)$.
Moreover, $\Psi$ is continuous when restricted to the subset $L_{\ms{inc},0}(\L'_1\times[0,\infty))$ of $L_{\ms{inc}}(\L'_1\times[0,\infty))$ consisting of those $F$ with $\mu_F(\partial M)=0$, where 
$$M=\{(y,s)\in\L'_1\times[0,\infty):s\le |y|^{-\alpha}\}$$
and  $\mc{K}\subset L_{\ms{inc}}(\L'_1\times [0,\infty))$ denotes the family of all $[0,\infty)$-valued, bounded, increasing and right-continuous functions.
 Since the Lebesgue measure of $\partial M$ is $0$, the rate function from the LDP of Proposition~\ref{genLDP} is infinite on the complement of $\mc{K}$. Hence,~\cite[Lemma 4.1.5]{dz98} shows that 
the random fields $\big\{|\L_\t'|^{-1}Y^{*,\t}(\cdot,\cdot)\}_{\t<1}$ also satisfy an LDP on $\mc{K}$.  Hence, the contraction principle  applies and it remains to identify the rate function. That is, we need to show that 
\begin{align*}
&\inf_{\substack{F\in \mc{K} \\ G(\cdot)=F(\cdot\one_M)}}\int_{\L'_1}\inf_{\Q}h(\Q|\P)+\int_0^\infty h\Big(f(y,s)|\frac{\d}{\d s}\mu_\Q(s)\Big)\d s\d y\\
&\quad=\int_{\L'_1}\inf_{\Q}h(\Q|\P)+ h\big(g(y)|\mu_\Q(|y|^{-\alpha})\big)\d y,
\end{align*}
where $f=\partial F/(\partial y\partial s)$ and $g=\partial G/\partial y$ denote the Radon-Nikodym derivatives of $F$ and $G$, respectively. Note that if $G$ was not absolutely continuous, then neither could be $F$, so that the left-hand side would be infinity.
We show that the equality arises as a consequence of two inequalities. First, we consider the direction $\ge$. As in the proof of Proposition \ref{genLDP}, an application of Jensen's inequality implies that
%Here, it suffices to show that for any $F$,
$$\int_{0}^{|y|^{-\alpha}} h\big(f(y,s)|\frac{\d}{\d s}\mu_\Q(s)\big)\d s\ge h\Big(\int_{0}^{|y|^{-\alpha}}f(y,s)\d s|\mu_\Q(|y|^{-\alpha})\Big).$$
The right-hand side is equal to $h(g(y)|\mu_\Q(|y|^{-\alpha}))$ if $G(\cdot)=F(\cdot\one_M)$.
%Jensen's inequality gives that 
%\begin{align*}
%\int_1^{|y|^{-\alpha}} 
%h\Big(\frac{\partial F}{\partial y\partial s}(y,s)|\frac{\d}{\d s}\mu_\Q(s)\Big)&= \int_1^{|y|^{-\alpha}} h\Big(\frac{\partial F}{\partial y\partial s}(y,s)/\frac{\d}{\d s}\mu_\Q(s)\Big)\frac{\d}{\d s}\mu_\Q(s) \d s\\
%&\ge \mu_\Q(|y|^{-\alpha})h\Big(\frac{1}{\mu_\Q(|y|^{-\alpha})}\int_1^{|y|^{-\alpha}} \frac{\d}{\d s}\mu_\Q(s)\frac{\partial F}{\partial y\partial s}(y,s)/\frac{\d}{\d s}\mu_\Q(s) \d s\Big)\\
%&\ge \mu_\Q(|y|^{-\alpha})h\Big(\frac{1}{\mu_\Q(|y|^{-\alpha})}\frac{\partial F}{\partial y}(y,|y|^{-\alpha}) \d s\Big),
%\end{align*}
%as required.

The other direction is more involved. First, we proceed as in the proof of Proposition \ref{genLDP} and note that the right-hand side can be approximated using a suitable discretization. To be more precise, let $\{\rho(l)\}_{l=1}^{2^{dk}}$ be a subdivision of $\L_1'$ into congruent cubes of side length $\delta(k)=|\L_1'|^{1/d}2^{-k}$. The point in the $l$-th cube which minimizes the distance to the origin will be denoted by $y_{k,l}$. In the first step of the discretization, we replace the expression $\mu_\Q(|y|^{-\alpha})$ by $\mu_\Q(|y_{k,l}|^{-\alpha})$.

\begin{lemma}
\label{cpApp1Lem}
\begin{align*}
&\limsup_{k\to\infty}\sum_{l=1}^{2^{dk}}\int_{\rho(l)}\inf_{\Q\in\mc{P}_\theta}h(\Q|\P)+h\big(g(y)\big|\mu_\Q(|y_{k,l}|^{-\alpha})\big)\d y\\
&\quad\le\int_{\L'_1}\inf_{\Q\in\mc{P}_\theta}h(\Q|\P)+ h\big(g(y)\big|\mu_\Q(|y|^{-\alpha})\big)\d y.
\end{align*}
\end{lemma}
\begin{proof}
First, note that for every $l\in\{1,\ldots,2^{dk}\}$, $y\in\rho(l)$ and $\Q\in\mc{P}_\theta$ we get that 
\begin{align*}
&h\big(g(y)\big|\mu_\Q(|y_{k,l}|^{-\alpha})\big)-h\big(g(y)\big|\mu_\Q(|y_{}|^{-\alpha})\big)\\
&\quad\le g(y)\log\frac{\mu_\Q(|y_{}|^{-\alpha})}{\mu_\Q(|y_{k,l}|^{-\alpha})}+|\mu_\Q(|y_{k,l}|^{-\alpha})-\mu_\Q(|y|^{-\alpha})|\\
&\quad\le|\mu_\Q(|y_{k,l}|^{-\alpha})-\mu_\Q(|y|^{-\alpha})|
\end{align*}
where the last inequality follows from the choice of $y_{k,l}$.  In particular, the right-hand side is always bounded above by $1$. Moreover, for $\varepsilon>0$ we let $A_{\varepsilon}^{}=\{l\in\{1,\ldots,2^{dk}\}:\, \min_{y\in\rho(l)}|y|<\varepsilon\}$ denote the set of indices of cubes that are close to the origin. Then, the Lipschitz assumption implies that for every $l\not\in A_{\varepsilon}$, $y\in \rho(l)$ and $\Q\in\mc{P}_\theta$,
\begin{align}
\label{muDiffEq}
|\mu_\Q(|y_{k,l}|^{-\alpha})-\mu_\Q(|y|^{-\alpha})|\le N\alpha|y_{k,l}|^{-\alpha-1}|y-y_{k,l}|\le N\alpha\varepsilon^{-\alpha-1}\sqrt{d}\delta(k)^{-1}.
\end{align}
 Hence,
\begin{align*}
&\sum_{l=1}^{2^{dk}}\int_{\rho(l)}\inf_{\Q\in\mc{P}_\theta}h(\Q|\P)+h\big(g(y)\big|\mu_\Q(|y_{k,l}|^{-\alpha})\big)\d y\\
&\quad-\Big(\int_{\L'_1}\inf_{\Q\in\mc{P}_\theta}h(\Q|\P)+ h\big(g(y)\big|\mu_\Q(|y|^{-\alpha})\big)\d y\Big)\\
&\quad\le |\rho(1)|\#A_{\varepsilon}^{}+N\alpha\varepsilon^{-\alpha-1}\sqrt{d}\delta(k)\le 2^d\varepsilon^{d}+N\alpha\varepsilon^{-\alpha-1}\sqrt{d}\delta(k),
\end{align*}
provided that $k\ge1$ is sufficiently large. Since $\varepsilon>0$ was arbitrary, this completes the proof.
\end{proof}
The next lemma is proved similarly to Proposition \ref{genLDP} using
%In the next step, we proceed similarly to Lemma~\ref{iRepLem} and use 
Jensen's inequality and a discretization of the integral. We omit the proof.
\begin{lemma}
\label{cpApp2Lem}
Let $k\ge1$ and $1\le l\le 2^{dk}$ be arbitrary. Then,
\begin{align*}
&|\rho(l)|^{-1}\int_{\rho(l)}\inf_{\Q\in\mc{P}_\theta}h(\Q|\P)+h\big(g(y)\big|\mu_\Q(|y_{k,l}|^{-\alpha})\big)\d y\\
&\quad\ge\inf_{\Q\in\mc{P}_\theta}h(\Q|\P)+h\big(|\rho(l)|^{-1}G(\rho(l))\big|\mu_\Q(|y_{k,l}|^{-\alpha})\big)\d y.
\end{align*}
\end{lemma}
%\begin{proof}
%By Jensen's inequality, we obtain that
%\begin{align*}
%\end{align*}
%as required.
%\end{proof}
Now that we have discretized the integral, we can define approximations $F^{(k)}$ to the desired function $F$. For this purpose, we first need to construct certain minimizers. Recall from Corollary~\ref{qContCor} that the function $\Q\mapsto\mu_\Q(|y_{k,l}|^{-\alpha})$ is continuous on every set of the form $\{\Q:\,h(\Q|\P)\le K\}$ for some  $K<\infty$. Therefore, the function
$$\Q\mapsto h(\Q|\P)+h\big(|\rho(l)|^{-1}G(\rho(l))\big|\mu_\Q(|y_{k,l}|^{-\alpha})\big)$$
is lower semicontinuous, and we let $\Q_{k,l}$ be one of its minimizers. 
%\begin{lemma}
%\label{minExLem}
%Let $k\ge1$ and $1\le l\le k$. Then, there exists a stationary random point field $\Q_{k,l}\in\mc{P}_\theta$ such that 
%$$h(\Q_{k,l}|\P)+h\big(|\rho(l)|^{-1}G(\rho(l))\big|\mu_{\Q_{k,l}}(|y_{k,l}|^{-\alpha})\big)=\inf_{\Q}h(\Q|\P)+h\big(|\rho(l)|^{-1}G(\rho(l))\big|\mu_\Q(|y_{k,l}|^{-\alpha})\big).$$
%\end{lemma}
%\begin{proof}
%Clearly, we may assume that $a<\infty$, where $a$ is the right-hand side of the asserted equality. Then, there exist $\Q_n\in\mc{P}_\theta$, $n\ge1$ such that 
%$$\lim_{n\to\infty}h(\Q_n|\P)+h\big(|\rho(l)|^{-1}G(\rho(l))\big|\mu_{\Q_n}(|y_{k,l}|^{-\alpha})\big)=a.$$
%Since the level sets of $h(\cdot|\P)$ are compact, we may assume that $\Q_n\to\Q_{k,l}$ for some $\Q_{k,l}\in\mc{P}_\theta$. Moreover, by Corollary~\ref{qContCor},
%$$\lim_{n\to\infty}\mu_{\Q_n}(|y_{k,l}|^{-\alpha})=\mu_{\Q_{k,l}}(|y_{k,l}|^{-\alpha}).$$
%Hence, an application of lower-semicontinuity of $h$ concludes the proof.
%\end{proof}
Now, define measurable functions $f^{(k)}:\L'_1\to[0,\infty]$, $k\ge1$ by 
\begin{align*}
f^{(k)}(y,s)=
\begin{cases}
\frac{|\rho(l)|^{-1}G(\rho(l))\frac{\d}{\d s}\mu_{\Q_{k,l}}(s)}{\mu_{\Q_{k,l}}(|y_{k,l}|^{-\alpha})}&\text{if $y\in\rho(l)$ and $s\le|y_{k,l}|^{-\alpha}$},\\
\frac{\d}{\d s}\mu_{\Q_{k,l}}( s)&\text{if $y\in\rho(l)$ and $s>|y_{k,l}|^{-\alpha}$.}
\end{cases}
\end{align*}
Here, we make the convention that the first line is equal to zero if $\mu_{\Q_{k,l}}(|y_{k,l}|^{-\alpha})=G(\rho(l))=0$ and is equal to infinity if $\mu_{\Q_{k,l}}(|y_{k,l}|^{-\alpha})=0$, but $G(\rho(l))\ne0$.
Furthermore, we let $F^{(k)}$ denotes the distribution function of the measure with density $f^{(k)}(y,s)$. Then, for every $y\in\rho(l)$,
\begin{align*}
&\inf_{\Q\in\mc{P}_\theta}h(\Q|\P)+\int_{0}^\infty h\Big(f^{(k)}(y,s)\big|\frac{\d}{\d s}\mu_{\Q}(s)\Big)\d s\\
&\quad\le h(\Q_{k,l}|\P)+\int_{0}^{|y_{k,l}|^{-\alpha}} h\Big(\frac{|\rho(l)|^{-1}G(\rho(l))\frac{\d}{\d s}\mu_{\Q_{k,l}}(s)}{\mu_{\Q_{k,l}}(|y_{k,l}|^{-\alpha})} \big|\frac{\d}{\d s}\mu_{\Q_{k,l}}(s)\Big)\d s \\
&\quad=h(\Q_{k,l}|\P)+h\big(|\rho(l)|^{-1}G(\rho(l))\big|\mu_{\Q_{k,l}}(|y_{k,l}|^{-\alpha})\big).
\end{align*}

By the goodness of the rate function in Proposition~\ref{genLDP}, the functions $(F^{(k)})_{k\ge1}$ have an accumulation point, and lower-semicontinuity therefore gives that 
\begin{align*}
\int_{\L_1'}\inf_{\Q\in\mc{P}_\theta}h(\Q|\P)+\int_0^\infty h\big(f(y,s)|\mu_\Q(s)\big)\d s\le \int_{\L'_1}\inf_{\Q\in\mc{P}_\theta}h(\Q|\P)+ h\big(g(y)|\mu_\Q(|y|^{-\alpha})\big)\d y.
\end{align*}
Hence, it remains to show that the measure induced by $G(\cdot)$ coincides with the measure induced by $F(\cdot \one_M)$. In order to prove this claim, we first show that $F^{(k)}(M^{(k)}\setminus M)$ tends to zero as $k$ tends to infinity, where 
$M^{(k)}=\{(y,s)\in\L'_1\times[0,\infty):\,y\in\rho(l) \text{ and }s\le |y_{k,l}|^{-\alpha}\}$
\begin{lemma}
\label{mDiffSet}
The expression $F^{(k)}(M^{(k)}\setminus M)$ tends to zero as $k$ tends to infinity.
\end{lemma}
\begin{proof}
First, observe that $F^{(k)}(M^{(k)}\setminus M)$ can be expressed as 
\begin{align*}
F^{(k)}(M^{(k)}\setminus M)=\sum_{l=1}^{2^{dk}}G(\rho(l))|\rho(l)|^{-1}\int_{\rho(l)}\frac{\mu_{\Q_{k,l}}(|y_{k,l}|^{-\alpha})-\mu_{\Q_{k,l}}(|y|^{-\alpha})}{\mu_{\Q_{k,l}}(|y_{k,l}|^{-\alpha})}\d y.
\end{align*}
Now, for $\varepsilon>0$ introduce the set 
$$A_{\varepsilon}^{}=\{l\in\{1,\ldots,2^{dk}\}:\, \min_{y\in\rho(l)}|y|<\varepsilon \text{ or }\max_{y\in\rho(l)}|y|^{-\alpha}>ws_{\ms{min}}-\varepsilon\}$$
of indices whose associated cubes are far away from the origin and the boundary of the ball $B_{(ws_{\ms{min}})^{-1/\alpha}}(o)$.
%Note that we can neglect contributions both from a $\delta$-environment around $0$ and around the boundary of $B_1(o)$. 
Hence, we arrive at 
\begin{align}
\label{fkEq}
F^{(k)}(M^{(k)}\setminus M)&=\alpha\varepsilon^{-\alpha-1} N\sqrt{d}\delta(k)\sum_{l\not\in A_{\varepsilon}} \frac{G(\rho(l))}{\mu_{\Q_{k,l}}(|y_{k,l}|^{-\alpha})}+r(\varepsilon),
\end{align}
where $r(\varepsilon)$ tends to zero as $\varepsilon$ tends to zero. Since the sum above consists of at most $2^{dk}$ summands, it suffice to consider those $l\not\in A_\varepsilon$ that satisfy $|\rho(l)|^{-1}{G(\rho(l))}>{\mu_{\Q_{k,l}}(|y_{k,l}|^{-\alpha})}$. Now, note that if $l\not\in A_\varepsilon$, then $\mu_{\P}(|y_{k,l}|^{-\alpha})\ge 1/K$ for some sufficiently large $K=K(\varepsilon)$ not depending on $k,l$. We claim that also ${\mu_{\Q_{k,l}}(|y_{k,l}|^{-\alpha})}\ge 1/K$. Once this is shown, the proof is complete.

Suppose that ${\mu_{\Q_{k,l}}(|y_{k,l}|^{-\alpha})}<1/K$. First, if $|\rho(l)|^{-1}G(\rho(l))\ge \mu_\P(|y_{k,l}|^{-\alpha})$, then 
$$h\big(|\rho(l)|^{-1}G(\rho(l))|\mu_\P(|y_{k,l}|^{-\alpha})\big)< h\big(|\rho(l)|^{-1}G(\rho(l))|\mu_{\Q_{k,l}}(|y_{k,l}|^{-\alpha})\big),$$
which contradicts the minimality of $\Q_{k,l}$. Otherwise put $\Q^*=\lambda \P+(1-\lambda){\Q_{k,l}}$, where $\lambda\in[0,1)$ is chosen such that
$$\lambda \mu_{\P}(|y_{k,l}|^{-\alpha})+(1-\lambda) \mu_{\Q_{k,l}}(|y_{k,l}|^{-\alpha})=|\rho(l)|^{-1}G(\rho(l)).$$
Then, since the specific relative entropy $h$ (introduced in Section~\ref{intrSec}) is an affine function, 
\begin{align*}
&h(\Q^*|\P)+h\big(|\rho(l)|^{-1}G(\rho(l))|\mu_{\Q^*}(|y_{k,l}|^{-\alpha})\big)\\
&\quad=(1-\lambda)h(\Q_{k,l}|\P)\\
&\quad<h(\Q_{k,l}|\P)+h\big(|\rho(l)|^{-1}G(\rho(l))|\mu_{\Q_{k,l}}(|y_{k,l}|^{-\alpha})\big),
\end{align*}
which contradicts again the minimality of $\Q_{k,l}$.
\end{proof}

Now, we can complete the proof of Theorem~\ref{lowTauProbThm}.
\begin{proof}[Proof of Theorem~\ref{lowTauProbThm}]
By Lemmas~\ref{cpApp1Lem} and~\ref{cpApp2Lem}, it suffices to show that $G(f)=F(f\one_M)$ holds for any $f:\L'_1\to[0,\infty)$ of the form $f=\one_{\rho(l_0)}$ for some $1\le l_0\le 2^{dk_0}$. 
Now, we can argue as follows.
\begin{align*}
|F(f\one_M)-G(f)|\le\limsup_{k\to\infty} |F^{(k)}(f\one_{M^{(k)}})-G(f)|+\limsup_{k\to\infty}F^{(k)}(f\one_{M^{(k)}\setminus M}).
\end{align*}
Lemma~\ref{mDiffSet} shows that the second summand is zero. Moreover, by definition of $F^{(k)}$, 
$$F^{(k)}(f\one_{M^{(k)}})=G^{(k)}(f),$$
where $G^{(k)}$, $k\ge k_0$ denotes the measure with locally constant density $g^{(k)}(y)=|\rho(l)|^{-1}G(\rho(l))$. Hence, 
$$G^{(k)}(f)=G^{(k)}(\rho(l_0))=\sum_{\rho(l)\subset \rho(l_0)} G(\rho(l))=G(\rho(l_0))=G(f),$$
as required.
\end{proof}

%\subsection{Extension to uniform topology}
%\input{ut.tex}

%\section{Proof of Theorem~\ref{lowTauProbThm}}
%\input{oldC4.tex}

\subsection{Proof of Corollary~\ref{lowTauProbCor}}
\begin{proof}
The upper estimate is a direct consequence of the upper bound in Theorem~\ref{lowTauProbThm}
\begin{align*}
\limsup_{\t\to0}|\Lambda_\t'|^{-1}\log p_\t&=\limsup_{\t\to0}|\Lambda_\t'|^{-1}\log\P\big(|\Lambda_\t'|^{-1}Y^{\t}(t^{-\b}\cdot)=0\big)\cr
%&\le-\mc{I}(0)=-\int_{[-1/2,1/2]^d}\mc{I}_y(0)\d y\cr
&\le-\int_{\L_1'}\inf_{\Q\in\mc{P}_\theta}\big(h(\Q|\P)+\lambda_{\ms{R}}\Q(\Gamma(|y|^{-\alpha},o))\big)\d y.
%&\lim_{\t\to0}|\Lambda_\t'|^{-1}\log\E\exp\Big(-\lambda_{\ms{R}}\int_{\R^d} \Gamma(t^{-1}|y|^{-\alpha})\d y\Big)\\
%&=-\int_{\Lambda'_1}\inf_{\Q\in\mc{P}_\theta}\big(h(\Q|\P)+ \lambda_{\ms{R}}\Q(\Gamma(|y|^{-\alpha}))\big)\d y,
\end{align*}
For the lower estimate first note that
\begin{align*}
p_\t
%\lim_{\t\to0}|\Lambda_\t'|^{-1}\log\E\exp\Big(-\lambda_{\ms{R}}\int_{\R^d} \Gamma(\t^{-1}|y|^{-\alpha})\d y\Big)\cr
&=\E\exp\Big(-\lambda_{\ms{R}}\int_{\Lambda_\t'}\Gamma(\t^{-1}\ell(y),y)\d y\Big)
\end{align*}
where $\Gamma(a,y)=\E(q(a\I(y)^{-1})|X)$ is a non-local function of the transmitter process.
% which is also not translation-invariant. 
In order to be able to apply~\cite[Theorem  3.1]{georgii2}, we need to establish a translation-invariant setting using discretization of the integrand. To be more precise, we subdivide $\Lambda'_{\t}$ into $2^{dn}$ sub-cubes $\Lambda^{i}_\t$ of side length $2^{1-n}|ws_{\ms{min}}\t|^{-\beta}$ and let $y_i$ denote the corresponding element of the subcube $\Lambda^{i}_\t$ which is closest to the origin. Then
\begin{align*}
p_\t
&\ge\E\exp\Big(-\lambda_{\ms{R}}\sum_{i=1}^{2^{dn}}\int_{\Lambda^i_\t}\Gamma(\t^{-1}\ell(|y_i|),y)\d y\Big)
%&=2^{-dN}\sum_{i=1}^{2^{dN}}|\Lambda^i_\t|^{-1}\log\E\exp\Big(-\lambda_{\ms{R}}\int_{\Lambda^i_\t}\E(q(\t^{-1}|y_i|^{-\alpha}\I_b(y)^{-1})|\cdot)\d y\Big).
\ge\E\exp\Big(-\lambda_{\ms{R}}\sum_{i=1}^{2^{dn}}\int_{\Lambda^i_\t}\Gamma^b(\t^{-1}\ell(|y_i|),y)\d y\Big)
\end{align*}
where $\Gamma^b(a,y)=\E(q(a\I^b(y)^{-1})|X)$. 
Further let $X^{\ms{per},i}$ be the configuration obtained after extending the configuration of the marked Poisson point process $X$ in the subcube $\Lambda^i_\t$ periodically in the entire Euclidean space $\R^d$. 
The error made replacing $\sum_{i=1}^{2^{dn}}\int_{\Lambda^i_\t}\Gamma^b(\t^{-1}\ell(y_i),y)$ by $\sum_{i=1}^{2^{dn}}\int_{\Lambda^i_\t}\Gamma^{\ms{per},b}(\t^{-1}\ell(y_i),y)$ is negligible in the large deviation principle where 
$\Gamma^{\ms{per},b}(a,y)=\E(q(a\I^b(y)^{-1})|X^{\ms{per},i})$, indeed
%$X^{\ms{per}}:=(X^{\ms{per},i})_{i\in\{1,\dots,2^{dn}\}}$, indeed 
\begin{align}\label{per estimate}
&\int_{\Lambda^i_\t}[\Gamma^b(\t^{-1}\ell(y_i),y)-\Gamma^{\ms{per},b}(\t^{-1}\ell(y_i),y)]\d y\cr
&\le N(w^2t)^{-1}\int_{\Lambda^i_\t}[\I^b(o,y,X)-\I^b(o,y,X^{\ms{per},i})]\d y\cr
&\le N(s^2_{\ms{min}}w^2t)^{-1}\Big[\sum_{\substack{X_j\in X \\ X_j\not\in\Lambda^i_\t, X_j\in\Lambda^i_{\t,b}}}\int_{\Lambda^i_\t}\ell_b(|X_j-y|)\d y+\sum_{\substack{X_j\in X^{\ms{per},i} \\ X_j\not\in\Lambda^i_\t, X_j\in\Lambda^i_{\t,b}}}\int_{\Lambda^i_\t}\ell_b(|X_j-y|)\d y\Big]\cr
&\le N(s^2_{\ms{min}}w^2t)^{-1}[X(\Lambda^i_{\t,b}\setminus\Lambda^i_\t)+X^{\ms{per},i}(\Lambda^i_{\t,b}\setminus\Lambda^i_\t)]\int_{B_{b}(o)}\ell_b(|y|)\d y,
\end{align}
where $\Lambda^i_{\t,b}$ denotes the volume $\Lambda^i_{\t}$ joined with its $b$-boundary. 
%In the last line we used that integrals become bigger if $X_j$ lies on the boundary of $\Lambda^i_\t$ and by the boundedness of $\ell_b$ for small $\t$ 
Consequently, for all $\e>0$,
\begin{align*}
\E\exp&\Big(-\lambda_{\ms{R}}\sum_{i=1}^{2^{dn}}\int_{\Lambda^i_\t}\Gamma^b(\t^{-1}\ell(|y_i|),y)\d y\Big)\cr
&\ge e^{-\e|\Lambda_t'|}\E\exp\Big(-\lambda_{\ms{R}}\sum_{i=1}^{2^{dn}}\int_{\Lambda^i_\t}\Gamma^{\ms{per},b}(\t^{-1}\ell(|y_i|),y)\d y\Big)\cr
&\hspace{1cm}-\P\Big(\lambda_{\ms{R}}\sum_{i=1}^{2^{dn}}\int_{\Lambda^i_\t}\Gamma^b(\t^{-1}\ell(y_i),y)-\Gamma^{\ms{per},b}(\t^{-1}\ell(y_i),y)\d y\ge\e|\Lambda_\t'|\Big).
\end{align*}
By equation \eqref{per estimate}, the second line is bounded from below by 
\begin{align*}
-2\P\Big(\lambda_{\ms{R}}2N(s^2_{\ms{min}}w^2t)^{-1}\int_{B_{b}(o)}\ell_b(|y|)\d y X(\Lambda^1_{\t,b}\setminus\Lambda^1_\t)
\ge\e|\Lambda_\t^1|\Big).
\end{align*}
But this goes to zero on an exponential scale infinitely fast by 
%\begin{align*}
%&\E\exp\Big(-\lambda_{\ms{R}}\sum_{i=1}^{2^{dn}}\int_{\Lambda^i_\t}\Gamma_b(\t^{-1}\ell(|y_i|),y)\d y\Big)\cr
%&\ge \exp(-\e|\Lambda_t'|)\E\exp\Big(-\lambda_{\ms{R}}\sum_{i=1}^{2^{dn}}\int_{\Lambda^i_\t}\Gamma_b^{\ms{per}}(\t^{-1}\ell(|y_i|),y)\d y\Big)-\P\Big(\lambda_{\ms{R}}\sum_{i=1}^{2^{dn}}\int_{\Lambda^i_\t}\Gamma_b(\t^{-1}\ell(y_i),y)-\Gamma_b^{\ms{per}}(\t^{-1}\ell(y_i),y)\ge\e|\Lambda_\t'|
%\Big)\cr
%&\ge\E\exp\Big(-\lambda_{\ms{R}}\sum_{i=1}^{2^{dn}}\int_{\Lambda^i_\t}\Gamma_b^{\ms{per}}(\t^{-1}\ell(|y_i|),y)\d y\Big)
%\exp\Big(-\lambda_{\ms{R}}N(s^2_{\ms{min}}w^2t)^{-1}\sum_{i=1}^{2^{dn}}|\Lambda^i_\t|[X(\Lambda^i_{\t,b}\setminus\Lambda^i_\t)+X^{\ms{per},i}(\Lambda^i_{\t,b}\setminus\Lambda^i_\t)]\Big)
%\end{align*}
Lemma~\ref{erosionLem}. 
%
%this goes to zero
%\begin{align*}
%\sum_{i=1}^{2^{dN}}&\int_{\Lambda^i_\t}|\Gamma_{y_i}(\t^{-1}|y|^{-\alpha})-\Gamma^{\ms{per}}_{y_i}(\t^{-1}|y|^{-\alpha})|\d y\cr
%&\le\sum_{i=1}^{2^{dN}}\int_{\Lambda^i_\t}[\Gamma_{y_i}(\t^{-1}|y|^{-\alpha})-\Gamma^{\ms{per}}_{y_i}(\t^{-1}|y|^{-\alpha})]\d y\cr
%\end{align*}
%
%\begin{align*}
%|\Lambda_\t'|^{-1}\log p_\t
%&\ge|\Lambda_\t'|^{-1}\log\E\prod_{i=1}^{2^{dN}}\exp\Big(-\lambda_{\ms{R}}\int_{\Lambda^i_\t}\E(q(\t^{-1}|y_i|^{-\alpha}\I_b(y)^{-1})|\cdot)\d y\Big)\cr
%&=2^{-dN}\sum_{i=1}^{2^{dN}}|\Lambda^i_\t|^{-1}\log\E\exp\Big(-\lambda_{\ms{R}}\int_{\Lambda^i_\t}\E(q(\t^{-1}|y_i|^{-\alpha}\I_b(y)^{-1})|\cdot)\d y\Big).
%\end{align*}
Hence, 
\begin{align*}
\lim_{\t\to0}|\L'_t|^{-1}&\log\E\exp\Big(-\lambda_{\ms{R}}\sum_{i=1}^{2^{dn}}\int_{\Lambda^i_\t}\Gamma^b(\t^{-1}\ell(y_i),y)\d y\Big)\cr
&\ge\lim_{\t\to0}|\L'_t|^{-1}\log\E\exp\Big(-\lambda_{\ms{R}}\sum_{i=1}^{2^{dn}}\int_{\Lambda^i_\t}\Gamma^{\ms{per},b}(\t^{-1}\ell(y_i),y)\d y\Big)\cr
&=2^{-dn}\sum_{i=1}^{2^{dn}}\lim_{\t\to0}|\L_t^i|^{-1}\log\E\exp\Big(-\lambda_{\ms{R}}\int_{\Lambda^i_\t}\Gamma^{\ms{per},b}(\t^{-1}\ell(y_i),y)\d y\Big),
\end{align*}
where we used the independence of the $X^{\ms{per},i}$ with respect to $i$ in the second line. Now we are in the position to apply~\cite[Theorem  3.1]{georgii2} and write
\begin{align*}
\lim_{\t\to0}|\L'_t|^{-1}\log&\E\exp\Big(-\lambda_{\ms{R}}\sum_{i=1}^{2^{dn}}\int_{\Lambda^i_\t}\Gamma^{\ms{per},b}(\t^{-1}\ell(y_i),y)\d y\Big)\cr
&\ge-2^{-dn}\sum_{i=1}^{2^{dn}}\inf_{\Q\in\mc{P}_\theta}\big(h(\Q|\P)+\lambda_{\ms{R}} \Q(\Gamma^b(|y_i|^{-\a},o))\big),
\end{align*}
where we also used the continuity of $\Gamma^{\ms{per},b}$ ensured by the truncation of the interference. Notice that 
\begin{align*}
%\label{b_away}
\limsup_{b\to\infty}\inf_{\Q\in\mc{P}_\theta}\big(h(\Q|\P)+\lambda_{\ms{R}} \Q(\Gamma^b(|y_i|^{-\a},o))\big)
\le\inf_{\Q\in\mc{P}_\theta}\big(h(\Q|\P)+\lambda_{\ms{R}} \Q(\Gamma(|y_i|^{-\a},o))\big).
\end{align*}
Indeed, let $\Q_0$ be a minimizer of the right hand side, then 
\begin{align*}
\limsup_{b\to\infty}\inf_{\Q\in\mc{P}_\theta}\big(h(\Q|\P)+\lambda_{\ms{R}} \Q(\Gamma^b(|y_i|^{-\a},o))\big)
&\le h(\Q_0|\P)+\lambda_{\ms{R}} \limsup_{b\to\infty}\Q_0(\Gamma^b(|y_i|^{-\a},o))
\end{align*}
and it suffices to show that
\begin{align*}
|\Q_0(\Gamma^b(|y|^{-\a},o))-\Q_0(\Gamma(|y|^{-\a},o))|\le N|y|^{-\a}w^{-2}\Q_0\Big(\sum_{X_j\not\subset\L_b}\ell(|X_j|)\Big)
\end{align*}
tends to zero as $b$ tends to infinity. But this is true since $\Q_0$ is a translation-invariant point process. 
In order to perform the large-$n$ limit, we have to show that 
\begin{align*}
y\mapsto\inf_{\Q\in\mc{P}_\theta}\big(h(\Q|\P)+\lambda_{\ms{R}} \Q(\Gamma(|y|^{-\a},o))\big)
\end{align*}
is continuous. But this is also true since 
\begin{align*}
&|\inf_{\Q\in\mc{P}_\theta}\big(h(\Q|\P)+\lambda_{\ms{R}} \Q(\Gamma(|y|^{-\a},o))\big)-\inf_{\Q\in\mc{P}_\theta}\big(h(\Q|\P)+\lambda_{\ms{R}} \Q(\Gamma(|x|^{-\a},o))\big)|\cr
&\le\sup_{\Q\in\mc{P}_\theta}\big|\lambda_{\ms{R}} \Q(\Gamma(|y|^{-\a},o))-\lambda_{\ms{R}} \Q(\Gamma(|x|^{-\a},o))\big|\le N w^{-1}\lambda_{\ms{R}}\big||y|^{-\a}-|x|^{-\a}\big|.
\end{align*}
This gives the result.
%
%\bigskip
%Since $\E(q(\t^{-1}|y_i|^{-\alpha}\I_b(y)^{-1})|\cdot)$ is local and translation invariant we can apply~\cite[Theorem  3.1]{georgii2}
%which gives
%\begin{align*}
%\lim_{\t\to0}|\Lambda_i|^{-1}&\log\E\exp\Big(-\lambda_{\ms{R}}\int_{\Lambda_i}\E(q(\t^{-1}|y_i|^{-\alpha}\I_b(y)^{-1})|\cdot)\d y\Big)\cr
%&=- \inf_{\Q\in\mc{P}_\theta}\big(h(\Q|\P)+\lambda_{\ms{R}} \Q(\E(q(|y_i|^{-\alpha}\I_b(o)^{-1}))\big)
%\end{align*}
%where the second summand is continuous as a function of $\Q$ in the local topology. Letting $N$ tend to infinity we thus have 
%\begin{align*}
%\lim_{\t\to0}|\Lambda_\t'|^{-1}\log p_\t
%%&\ge\lim_{\t\to0}|\Lambda_\t'|^{-1}\log\E\prod_{i=1}^{2^{dN}}\exp\Big(-\lambda_{\ms{R}}\int_{\Lambda_i}\E(q(\t^{-1}|y_i|^{-\alpha}\I_b(y)^{-1})|\cdot)\d y\Big)\cr
%&=-\int_{[-1/2,1/2]^d}\inf_{\Q\in\mc{P}_\theta}\big(h(\Q|\P)+\lambda_{\ms{R}} \Q(\E(q(|y|^{-\alpha}\I_b(o)^{-1}))\big)\d y.
%\end{align*}
%Now all that remains to prove is that the truncation of the interference can go to infinity.
%
%
%
%
%\begin{lemma}
%\label{Cor_Lem1}
%
%\end{lemma}
%\begin{proof}
\end{proof}

\section{Importance Sampling}
\label{impSampSec}

In this section, we show how the LDPs derived in Theorems~\ref{ldpThm} and~\ref{lowTauProbThm} can be used to devise an importance-sampling scheme improving the accuracy of basic Monte Carlo approaches for estimating the probability of observing unlikely configurations of connectable receivers.
%large-deviation result Corollary~\ref{lowTauProbCor} can be used to devise an importance-sampling that substantially the connection probability $p_\tau$. Indeed, as $\tau\to0$ the connection probability $p_\tau$ tends to $0$ exponentially fast in $\tau^{-\frac{d}{\alpha}}$. Hence, one should expect substantial accuracy gains when replacing naive Monte Carlo approaches by more refined techniques that are specifically designed for the problem of estimating rare-event probabilities.
Theorems~\ref{ldpThm} and~\ref{lowTauProbThm} imply that such probabilities generally tend to zero exponentially quickly, so that basic Monte Carlo estimators perform poorly.

The general heuristic for devising importance-sampling schemes is the following. Instead of sampling the transmitters according to their true distribution, the simulation is performed by using a modified law under which the considered rare event is more likely. An appropriate reweighting using likelihood ratios ensures the unbiasedness of the new estimator. For a more detailed discussion of the general technique of importance sampling, we refer to the textbooks~\cite{glynn,handbook}.
%Then, in order to obtain an unbiased estimator, a reweighting using likelihood ratios needs to be performed.

In principle, Theorems~\ref{ldpThm} and~\ref{lowTauProbThm} provide precise descriptions of the asymptotically exponentially optimal change of measure, in the sense that the modified law of transmitters should be given by suitable Gibbs point processes. However, as these distributions just arise as minimizers of fairly complicated functionals, it is difficult to use them for computational purposes. Still, by performing this minimization in the restricted class of Poisson point processes, we can achieve substantial accuracy benefits. 

We only provide a proof-of-concept for the use of importance sampling, and therefore assume a specific parameter constellation in the following. First, we fix $d=2$, $w=\lambda_{\ms{R}}=\lambda_{\ms{T}}=1$ and assume that the path-loss function is given by $\ell(r)=r^{-4}$. Moreover, we assume that there is no random environment, and that transmission powers and fading random variables are constant and equal to $1$. Note that this choice is not covered by the assumptions for Theorems~\ref{ldpThm} and~\ref{lowTauProbThm}. Nevertheless, our simulation results illustrate that variance reduction through importance sampling also hold under weaker conditions than the ones assumed in Theorems~\ref{ldpThm} and~\ref{lowTauProbThm}.
%can be achieved by performing importance sampling. 

\subsection{Importance sampling related to Theorem~\ref{ldpThm}}
\label{impSamp1Seq}
Since we have assumed that there is no random environment and that transmission powers and fading variables are constant, the minimization in the rate function of Theorem~\ref{ldpThm} is performed only over stationary point processes of transmitters and receivers. As mentioned above, this minimization is intractable in its full generality. Nevertheless, in this section, we show that if minimization is performed only in the class of Poisson point processes, then the problem becomes tractable. In fact, we provide an example problem, where the minimization can be reduced to a standard two-dimensional constrained minimization problem, where the constraint is given in terms of certain special functions. The disadvantage of this approach is that solving the minimization problem in a restricted class of point process will not automatically lead to good choices for the importance sampling. This will become apparent from the simulation results discussed below.

We assume that $\t=1$ and consider events of the form 
$$A_{n,a}=\Big\{\tfrac{1}{|\L_n|}\sum_{X_i\in\L_n}\#Y^{(i)}<a\Big\},$$ 
i.e., the event that the (spatially) averaged number of connectable receivers associated with transmitters in the cube $\L_n$ is less than $a$.  Now, we explain how to implement an importance-sampling scheme based on the LDP. A related importance-sampling scheme for a Poisson point process on the real line has already been considered in~\cite{giesecke}, but for the convenience of the reader, we present some details in our situation.

In order to estimate the probability of the event $A_{n,a}$, we simulate the Poisson point processes with new intensities $\mu_{\ms{R}}>0$ and $\mu_{\ms{T}}>0$ in $\Lambda_n$. Then, the likelihood ratio of Poisson point processes with intensity $1$ with respect to these point processes is given by 
$$\exp\big(|\L_n|(\mu_{\ms{R}}-1)+|\L_n|(\mu_{\ms{T}}-1)\big)\mu_{\ms{R}}^{-X(\L_n)}\mu_{\ms{T}}^{-Y(\L_n)}.$$
%where we recall that $X$ and $Y$ are the Poisson point processes of transmitters and receivers, respectively.
Hence, an unbiased estimator for $\P(A_{n,a})$ is given by
$$\widehat{p}_{n,a,\mu_{\ms{T}},\mu_{\ms{R}}}=\exp\big(|\L_n|(\mu_{\ms{R}}-1)+|\L_n|(\mu_{\ms{T}}-1)\big)\mu_{\ms{R}}^{-X(\L_n)}\mu_{\ms{T}}^{-Y(\L_n)}\one_{A_{n,a}}.$$
%where $X$ and $Y$ are Poisson point processes of intensity $\mu_{\ms{T}}$ and $\mu_{\ms{R}}$ on $\Lambda_n$. 
Note that in order to take into account edge effects, we also generate transmitters and receivers with the unmodified intensity in a small environment around $\L_n$. We can obtain estimates $\widehat{p}$ and $\widehat{v}$ of the expectation and variance of $\widehat{p}_{n,a,\mu_{\ms{T}},\mu_{\ms{R}}}$ by considering the sample average and variance of $N\ge1$ independent copies generated using Monte Carlo simulation.

This leaves the question as of how to find good choices for $\mu_{\ms{R}}$ and $\mu_{\ms{T}}$. In a first attempt, choose these parameters to minimize the large-deviation rate function appearing in Theorem~\ref{ldpThm}. If $\Q\in\mc{P}_\theta$ the distribution of independent Poisson point processes of  receivers and transmitters with intensities $\mu_{\ms{R}}$ and $\mu_{\ms{T}}$, then the relative entropy $h(\Q|\P)$ is given by the formula
\begin{align}
\label{pairEnt}
(\mu_{\ms{R}}\log \mu_{\ms{R}}-\mu_{\ms{R}}+1)+(\mu_{\ms{T}}\log \mu_{\ms{T}}-\mu_{\ms{T}}+1).
\end{align}
Hence, to determine the optimal intensities $(\lambda_{\ms{R},\ms{opt}},\lambda_{\ms{T},\ms{opt}})$ according to Theorem~\ref{ldpThm}, we need to minimize~\eqref{pairEnt} under the constraint
\begin{align}
\label{const1}
\Q^*(\#Y^{(o)})<a.
\end{align}
Next, we express the constraint~\eqref{const1} in terms of certain special functions. First, by Campbell's theorem, 
\begin{align*}
\Q^*(\#Y^{(o)})&=\mu_{\ms{T}}\mu_{\ms{R}}\int_{B_1(o)}\Q^*\Big(\frac{|x|^{-4}}{1+\sum_{i\ge1}|X_i|^{-4}}\ge 1\Big) \d x\\
&=\mu_{\ms{T}}\mu_{\ms{R}}2\pi\int_{0}^1r\P\Big(\sum_{i\ge1}|X_i|^{-4}\le\mu_{\ms{T}}^2(r^4-1)\Big) \d r,
\end{align*}
where in the last line we used that scaling by $1/\sqrt{\mu_\ms{T}}$ transforms a Poisson point process with intensity $1$ to a Poisson point process with intensity $\mu_{\ms{T}}$. 
Moreover, $\sum_{i\ge1}|X_i|^{-4}$ is distributed according to an inverse gamma distribution with parameters $0.5$ and $\pi^3/4$. In particular, $\P(\sum_{i\ge1}|X_i|^{-4}\le s)=\pi^{-1/2}\gamma(1/2,-\pi^3/(4s))$, where $\gamma(\cdot,\cdot)$ denotes the incomplete gamma function. Now it is easy to check that
\begin{align*}
\Q^*(\#Y^{(o)})&=\mu_{\ms{R}}\mu_{\ms{T}}2\pi\int_{0}^1r \gamma\big(1/2, \pi^3/(4\mu_{\ms{T}}^2(r^4-1))\big) \d r\\
&=\mu_{\ms{R}}\mu_{\ms{T}}\pi\exp\big(\pi^3/(4\mu_{\ms{T}}^2)\big)\ms{erfc}\big(\pi^{3/2}/(2\mu_{\ms{T}})\big),
\end{align*}
where $\ms{erfc}$ denotes the complimentary error function.
For instance, if we choose $a=0.5$, then $(\lambda_{\ms{R},\ms{opt}},\lambda_{\ms{T},\ms{opt}})\approx(0.832,0.984)$.

In order to assess the actual accuracy improvements that can be achieved with this importance sampling-scheme, we performed a prototypical Monte Carlo analysis. We fixed $a=0.5$, $n=25$ and performed $N=1,000,000$ simulation runs. We consider three different parameter choices for the importance sampling intensities $(\mu_{\ms{R}},\mu_{\ms{T}})$. First, we consider the case of basic Monte Carlo simulation, that is $(\mu_{\ms{R}},\mu_{\ms{T}})=(1,1)$. Second, we take the intensities that are obtained from the large-deviation analysis performed above, i.e., $(\mu_{\ms{R}},\mu_{\ms{T}})=(0.832,0.984)$. Third, we estimate $(\mu_{\ms{R}},\mu_{\ms{T}})$ from a simple cross-entropy scheme.  That is, we performed a pilot run of $100,000$ basic Monte Carlo simulations and determined the average intensities under the condition that the rare event occurs. This gives $(\mu_{\ms{R}},\mu_{\ms{T}})=(0.892,0.989)$. We refer the reader to~\cite{handbook} for further details on the general cross-entropy technique.
The results for $\widehat{p}$ and $\widehat{v}$ are reported in Table~\ref{tab1}.

\begin{table}
\begin{center}
\begin{tabular}{lcc}
 $(\mu_{\ms{R}},\mu_{\ms{T}})$  & $\widehat{p}$  & $\widehat{v}$\\
\hline
$(1,1)$ & $3.31\times 10^{-4}$ & $3.31\times 10^{-4}$ \\
$(0.832,0.984)$           & $3.12\times 10^{-4}$ & $3.69\times 10^{-4}$ \\
$(0.892,0.989)$            & $3.29\times 10^{-4}$ & $5.10\times 10^{-5}$
\end{tabular}
\caption{Comparison of the simulation results for the expectation and variance of the considered importance sampling estimators with transmitter and receiver intensities $(\mu_{\ms{R}},\mu_{\ms{T}})$.}
\label{tab1}
\end{center}
\end{table}

In particular, we would like to draw the attention to an important observation:  The estimator that is obtained as the solution of the optimization based on our large-deviation principle actually has a higher variance than the basic Monte Carlo estimator. Given the close relation between large-deviation theory and asymptotically optimal change of measures, this might come as a surprise at first sight. However, since we performed our optimization not in the full class of stationary point processes, but only considered Poisson point processes, the simulation output does not contradict this intuition. In fact, considered from a different perspective, the simulation results provide evidence that the optimal change of measure is rather far (in the Kullback-Leibler distance) from being a Poisson point process. In contrast, performing the change of measure with the intensities obtained from the pilot run shows that for the considered example, a more than seven-fold variance reduction can be achieved.

The discussion in the previous paragraph raises the legitimate question whether the change of measures deduced from the large-deviation result are of any practical use for importance sampling? Indeed, in the example described above, the intensities that lead to the seven-fold decrease in variance could be found without reference to the LDP, namely by an `educated guess' (or rather `cross-entropy').  Nevertheless, when considering importance sampling in the setting of Corollary~\ref{lowTauProbCor} finding a good importance sampling change of measure would involve `guessing' a continuous family of parameters, which is substantially more involved than what we have done above. In contrast, a simple analysis of the large-deviation rate function provides immediately a useful heuristic for the shape of the curve.

\subsection{Importance sampling related to Theorem~\ref{lowTauProbThm}}
Finally, we investigate importance-sampling techniques related to Theorem~\ref{lowTauProbThm}. We consider the specific setting of Corollary~\ref{lowTauProbCor}, i.e., estimation of the isolation probability $p_\t$ for small values of $\t$. Similar to the situation considered in Section~\ref{impSamp1Seq}, the full minimization problem is intractable, so that we restrict our attention to the class of homogeneous Poisson point processes. However, the situation is slightly different from the one considered in Section~\ref{impSamp1Seq}. Instead of globally optimizing a transmitter and receiver intensity, we now have the freedom to choose a different intensity for each point in $\L_1'$. Due to isotropy, this reduces to the task of choosing an optimal intensity $\lambda_{\ms{opt}}(r)$ for each $r\in[0,1]$.  This optimal intensity must minimize the following expression that can be derived from the variational characterization in Corollary~\ref{lowTauProbCor}:
\begin{align}
\label{pairEnt2}
\lambda_{\ms{opt}}(r)\log\lambda_{\ms{opt}}(r)-\lambda_{\ms{opt}}(r)+1+\P(r^{-4}\ge1+\lambda_{\ms{opt}}(r)^{2}\sum_{i\ge1}|X_i|^{-4}).
\end{align}
This is a standard minimization problem that can be solved by finding the roots of the derivative with respect to $\lambda_{\ms{opt}}(r)$. After some simplifications, we arrive at
\begin{comment}
That is, $\lambda_{\ms{opt}}(r)$ is a solution of 
$$\log\lambda_{\ms{opt}}(r)=2(r^{-4}-1)\lambda_{\ms{opt}}(r)^{-3}f(\lambda_{\ms{opt}}(r)^{-2}(r^{-4}-1)),$$
where $f$ denotes the density of an inverse gamma distribution with parameters $0.5$ and $\pi^3/4$. 

In particular, using the analytical expression for the density of an inverse gamma distribution, we arrive at 
\end{comment}
$$\log\lambda_{\ms{opt}}(r)=\frac{\pi}{\sqrt{r^{-4}-1}}\exp(\frac{\pi^3\lambda_{\ms{opt}}(r)^{2}}{4(r^{-4}-1)}).$$
This equation can be solved numerically; a plot of this solution is shown in Figure~\ref{intPlot}.
\begin{figure}[!htpb]
\centering
\includegraphics[width=6.6cm]{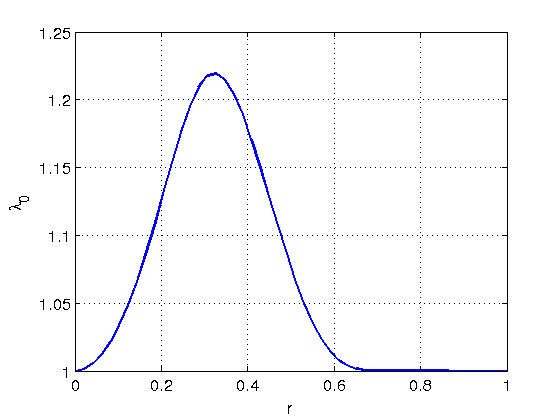}
\caption{Plot of the optimal density $\lambda_{\ms{opt}}(r)$ at distance $r$ from the origin}
\label{intPlot}
\end{figure}

As in the previous example, in order to assess the actual accuracy improvements for the estimation of $p_\t$ that can be achieved with this importance sampling-scheme, we performed a prototypical Monte Carlo analysis. We fixed $\t=0.002$ and performed $N=1,000,000$ simulation runs. 

\begin{table}
\begin{center}
\begin{tabular}{lcc}
   & $\widehat{p}$  & $\widehat{v}$\\
\hline
$\lambda(\cdot)\equiv1$ & $7.72\times 10^{-6}$ & $8.22\times 10^{-10}$ \\
$\lambda(\cdot)\equiv\lambda_{\ms{opt}}(\cdot)$            & $7.70\times 10^{-6}$ & $1.78\times 10^{-10}$
\end{tabular}
\caption{Comparison of the simulation results for the expectation and variance of the basic versus the importance sampling estimator.}
\end{center}
\end{table}

In contrast to the previous example, we see that the importance-sampling estimator derived from large-deviation theory provides substantial benefits. Indeed, the variance is reduced by approximately $78\%$. Furthermore, applying the cross-entropy technique for the present example would be substantially more involved than in the previous one. Indeed, instead of simply estimating two parameters, we would need to extract an entire curve from the pilot runs, so that proper statistical tools would be needed to estimate such a functional object from data.

\section*{Acknowledgments}
The authors thank W.~K\"onig for interesting discussions, and for proposing the topic of this article and the use of stationary empirical fields.
The authors thank M.~Renger for the hint to apply the contraction principle for deriving Theorem~\ref{lowTauProbThm}. We also thank T.~Brereton for pointing out the reference~\cite{giesecke}. This research was supported by the Leibniz group on Probabilistic Methods for Mobile Ad-Hoc Networks.

\bibliography{wias}
\bibliographystyle{abbrv}

\end{document}